\documentclass[11pt]{amsart}
\usepackage{amsmath,amsfonts,amscd,amssymb,epsf,euscript,tikz-cd,bm}
\usetikzlibrary{matrix,arrows}
\usepackage{amsxtra,mathrsfs,cancel}
\usepackage[normalem]{ulem}
\newtheorem{theorem}{Theorem}
\newtheorem{proposition}{Proposition}
\newtheorem{lemma}{Lemma}
\newtheorem{corollary}{Corollary}
\theoremstyle{definition}

\theoremstyle{remark}
\newtheorem{remark}{Remark}
\numberwithin{equation}{section}
\newcommand{\field}[1]{\ensuremath{\mathbb{#1}}}

\newcommand{\C}{\field{C}}

\newcommand{\HH}{\field{H}}

\newcommand{\R}{\field{R}}
\newcommand{\Z}{\field{Z}}

\newcommand{\complex}[1]{\mathsf{#1}} 

\newcommand{\SSS}{\complex{S}}
\newcommand{\BBB}{\complex{B}}
\newcommand{\KKK}{\complex{K}}
\newcommand{\AAA}{\complex{A}}
\newcommand{\CCC}{\complex{C}}

\newcommand{\Ka}{K\"{a}hler\:}
\newcommand{\Te}{Teichm\"{u}ller\:}
 \DeclareMathOperator{\id}{id}
 
\DeclareMathOperator{\Hom}{Hom} 

\DeclareMathOperator{\im}{Im} 
 
 \DeclareMathOperator{\PSL}{PSL}
 \DeclareMathOperator{\re}{Re}

\newcommand{\del}{\partial}
\newcommand{\delb}{\bar\partial}

\newcommand{\si}{\sigma}
\newcommand{\ga}{\gamma}
\newcommand{\Ga}{\Gamma}

\newcommand{\ihalf}{\frac{\sqrt{-1}}{2}}

\newcommand{\curly}[1]{\mathscr{#1}}

\newcommand{\cE}{\curly{E}}
\newcommand{\cF}{\curly{F}}

\newcommand{\cH}{\curly{H}}
\newcommand{\cI}{\curly{I}}

\newcommand{\cL}{\curly{L}}
\newcommand{\cM}{\curly{M}}

\newcommand{\cR}{\curly{R}}
\newcommand{\cS}{\curly{S}}

\newcommand{\vphi}{\varphi}
\newcommand{\bk}{\backslash}
\newcommand{\pa}{\partial}

\newcommand{\ov}{\overline}

\newcommand{\vep}{\varepsilon}

\newcommand{\z}{\bar{z}}

\begin{document}
\title[Potentials for WP and TZ metrics]{Potentials and Chern forms for Weil-Petersson and Takhtajan-Zograf metrics on moduli spaces}
\author{Jinsung Park}
\address{School of Mathematics, Korea Institute for Advanced Study, 207-43, Hoegiro 85, Dong-daemun-gu, Seoul, 130-722, Korea}
\email{jinsung@kias.re.kr}
\author{Leon A. Takhtajan}
\address{Department of Mathematics,
Stony Brook University, Stony Brook, NY 11794-3651, USA;
The Euler International Mathematical Institute, Saint Petersburg, Russia}
\email{leontak@math.sunysb.edu}
\author{Lee-Peng Teo}
\address{Department of Applied Mathematics, University of Nottingham Malaysia Campus,
Jalan Broga, 43500, Semenyih, Selangor, Malaysia}
\email{LeePeng.Teo@nottingham.edu.my}
\maketitle
\begin{abstract}
For the TZ metric on the moduli space $\cM_{0,n}$ of $n$-pointed rational curves, 
we construct a \Ka potential in terms of
the Fourier coefficients of the Klein's Hauptmodul.
 We define the space $\mathfrak{S}_{g,n}$ as holomorphic fibration $\mathfrak{S}_{g,n}\rightarrow\mathfrak{S}_{g}$ over the Schottky space $\mathfrak{S}_{g}$ of compact Riemann surfaces of genus $g$, where the fibers are configuration spaces of $n$ points.  For the tautological line bundles $\mathscr{L}_{i}$ over $\mathfrak{S}_{g,n}$ we define Hermitian metrics $h_{i}$ in terms of Fourier coefficients of a covering map $J$ of the Schottky domain. We define the regularized classical Liouville action $S$ and show that $\exp\{S/\pi\}$ is a Hermitian metric in the line bundle
$\cL=\otimes_{i=1}^{n}\cL_{i}$ over $\mathfrak{S}_{g,n}$. We explicitly compute the Chern forms of these Hermitian line bundles
$$c_{1}(\cL_{i},h_{i})=\frac{4}{3}\omega_{\mathrm{TZ},i},\quad c_{1}(\cL,\exp\{S/\pi\})=\frac{1}{\pi^{2}}\omega_{\mathrm{WP}}.$$
We prove that a smooth real-valued function $-\cS=-S+\pi\sum_{i=1}^{n}\log h_{i}$ on $\mathfrak{S}_{g,n}$, a potential for this special difference of WP and TZ metrics,
coincides with the renormalized hyperbolic volume of a corresponding Schottky $3$-manifold. We extend these results to the quasi-Fuchsian groups of type $(g,n)$.
\end{abstract}
\tableofcontents
\section{Introduction}

Weil introduced the Weil-Petersson (WP) metric on the moduli spaces of Riemann surfaces by using the Petersson inner product on the holomorphic cotangent spaces, the complex vector spaces of cusp forms of weight $4$. Ahlfors proved that the WP metric is K\"{a}hler and its Ricci, holomorphic sectional and scalar curvatures are all negative \cite{Ahl61,Ahl62}, and Wolpert found a closed formula for the Riemann tensor of the WP metric and obtained explicit bounds for its curvatures \cite{W86}.

In \cite{TZ87a,TZ87b} it was shown that for the moduli space $\cM_{0,n}$ of marked Riemann surfaces of type $(0,n)$, $n>3$ ($n$-pointed rational curves) and for the Schottky space $\mathfrak{S}_{g}$ of compact Riemann surfaces of genus $g>1$ the WP metric has global K\"{a}hler potential, the so-called \emph{classical Liouville action} (for precise definitions, see Sects.~\ref{Basic} and \ref{Liouv}).
In \cite{TZ88,TZ91} a new \Ka metric was introduced on the moduli space $\mathfrak{M}_{g,n}$ of Riemann surfaces of genus $g$ with $n>0$ punctures, $3g-3+n>0$. In \cite{Obitsu99,ObitsuW08,Weng01,W07} it was called Takhtajan-Zograf (TZ) metric (for its precise definition, see Sect.~\ref{Ka metrics}).  Unlike the WP metric, the curvature properties of the TZ metric are not known.

Here we present explicit formula for a \Ka potential $h_{i}$ of the $i$-th TZ metric on the moduli space $\cM_{0,n}$, $i=1,\dots,n$. Specifically, in Proposition \ref{TZ-potential} we prove that $h_{i}$ is expressed in terms of the first Fourier coefficients of Fourier expansions of the Klein's Hauptmodul $J$ at the cusps, introduced in \eqref{J-1}-\eqref{J-n}. The functions $h_{i}$ on $\cM_{0,n}$ provide explicit expressions for trivializations of the Hermitian metrics in the (holomorphically trivial) tautological line bundles $\cL_{i}$ on $\cM_{0,n}$, introduced in \cite{Weng01,W07}. Proposition \ref{TZ-potential} is the statement that the first Chern form of the Hermitian line bundle $\cL_{i}$ is $\dfrac{4}{3}\omega_{\mathrm{TZ},i}$, the symplectic form of the $i$-th TZ metric on $\cM_{0,n}$, $i=1,\dots,n$.

The function $H=h_{1}\dots h_{n-1}/h_{n}$ on $\cM_{0,n}$ determines a Hermitian metric in the line bundle $\lambda_{0,n}$ over the moduli space $\mathfrak{M}_{0,n}$ of type $(0,n)$ Riemann surfaces, introduced by Zograf \cite{Z89} (see Lemma \ref{Hermite metric} and Sect.~\ref{genus 0} for details). We show (see Corollary \ref{TZ--potential}) that on $\mathfrak{M}_{0,n}$
$$c_{1}(\lambda_{0,n},H)=\frac{4}{3}\omega_{\mathrm{TZ}},$$
where $\omega_{\mathrm{TZ}}=\omega_{\mathrm{TZ},i}+\cdots+\omega_{\mathrm{TZ},n}$ is the symplectic form of the TZ metric on $\mathfrak{M}_{0,n}$. Comparison with the known result (see \cite{TZ87a,Z89})
$$c_{1}(\lambda_{0,n},\exp\{S/\pi\})=\frac{1}{\pi^2}\omega_{\mathrm{WP}},$$
where $S$ is the classical Liouville action and $\omega_{\mathrm{WP}}$ is the symplectic form of the WP metric,
shows that a real-valued function $\cS=S-\pi H$ on $\frak{M}_{0,n}$ is a global \Ka potential for a special linear combination $\omega_{\mathrm{WP}}-\dfrac{4\pi^{2}}{3}\omega_{\mathrm{TZ}}$ of the WP and TZ metrics. 

We also study WP and TZ metrics on the deformation spaces of punctured Riemann surfaces of genus $g>1$. Namely, we introduce the Schottky space $\mathfrak{S}_{g,n}$ of type $(g,n)$ Riemann surfaces as a holomorphic fibration $\mathfrak{S}_{g,n}\rightarrow\mathfrak{S}_{g}$ whose fibers are configuration spaces of $n$ points (for details, see Sect.~\ref{S space}).  Denote by $J$ the corresponding covering map of the Schottky domain $\Omega$ and put $h_{i}=|a_i(1)|^2$, where $a_i(1)$ are the first Fourier coefficients of $J$  at the punctures $z_i$, $i=1,\ldots,n$, given by \eqref{F-J}. In Lemma \ref{h tautological} we prove that $h_{i}$ determine Hermitian metrics on the tautological line bundles $\mathscr{L}_{i}$ --- holomorphic line bundles dual to the vertical tangent bundle on $\mathfrak{S}_{g,n}$ along the fibers of the projection $p_{i}:\mathfrak{S}_{g,n}\rightarrow\mathfrak{S}_{g,n-1}$ which `forgets' the marked point $w_{i}$, $i=1,\dots,n$.

In Sect.~\ref{S domains} we define regularized classical Liouville action $S$ and prove that $\exp\{S/\pi\}$
determines a Hermitian metric in the holomorphic line bundle
$\mathscr{L}=\cL_{1}\otimes\cdots\otimes\cL_{n}$ over $\mathfrak{S}_{g,n}$ (see Lemma \ref{transform}). Sect.~\ref{s:potential} contains the main results of the paper. Thus in Sect.~\ref{TZ potential} we present explicit potentials for the TZ metrics on $\cM_{0,n}$, and in Theorems \ref{L-TZW} and \ref{S-TZWP} we explicitly compute canonical connections and Chern forms of the Hermitian line bundles $(\cL_{i},h_{i})$ and $(\cL,\exp\{S/\pi\})$. Namely, we show that
 \begin{align}
 c_{1}(\cL_{i},h_{i}) & =\frac{4}{3}\omega_{\mathrm{TZ},i},\quad i=1,\dots,n, \label{L-i}\\
c_{1}(\cL,\exp\{S/\pi\})& =\frac{1}{\pi^{2}}\omega_{\mathrm{WP}}. \label{L-WP}
\end{align}
Here $\omega_{\mathrm{WP}}$ and $\omega_{\mathrm{TZ}}$ are, respectively, symplectic forms of the WP and TZ metrics on $\mathfrak{S}_{g,n}$.

The statement that the first Chern class of the line bundles $\cL_{i}$ is $\frac{4}{3}\omega_{\mathrm{TZ},i}$ was proved in \cite{TZ91} at the level of cohomology classes and in \cite{Weng01,W07} at the level of Chern forms. Hermitian metrics $h_{i}$ in the line bundles $\cL_{i}$ on $\mathfrak{S}_{g,n}$ provide explicit expressions for the pullbacks of the Hermitian metrics in tautological line bundles over the moduli space $\cM_{g,n}$ of $n$-pointed curves of genus $g>1$, introduced in \cite{Weng01,W07}.

The quantity
$$
\cS=S-\pi \sum_{i=1}^n \log h_{i}
$$
is a smooth real-valued function on the Schottky space $\mathfrak{S}_{g,n}$. It follows from \eqref{L-i} and \eqref{L-WP} that $-\cS$ is a \Ka potential for a special linear combination of the WP and TZ metrics,
\begin{equation}\label{e:main result}
\delb\del \cS=-2\sqrt{-1}\left(\omega_{\mathrm{WP}}-\frac{4\pi^{2}}{3}\omega_{\mathrm{TZ}}\right),
\end{equation}
where $\del$ and $\delb$ are $(1,0)$ and $(0,1)$ components of the de Rham differential on $\mathfrak{S}_{g,n}$. This linear combination, with the overall factor $1/12\pi$, is precisely the one that appears in the local index theorem for families on punctured Riemann surfaces for $k=0,1$ in \cite[Theorem 1]{TZ91}. 

In Sect.~\ref{s:q-Fuchsian} we extend the approach in \cite{TT03} to quasi-Fuchsian groups of type $(g,n)$. Namely,  we define the classical Liouville action and  in Theorem  \ref{qF case}
prove that it is a \Ka potential of the WP metric on the quasi-Fuchsian deformation space. In Sect.~\ref{Holography} we study renormalized volumes of the corresponding Schottky and quasi-Fuchsian $3$-manifolds. In Theorem \ref{S holography} we prove that the renormalized hyperbolic volume of the corresponding Schottky $3$-manifold is related to the above-mentioned function $\cS$ and in Theorem \ref{qF holography} we prove that for quasi-Fuchsian $3$-manifolds it is related to the regularized Liouville action. These extend the results obtained in \cite{TT03}  to punctured Riemann surfaces.
\subsection*{Acknowledgements}
The work of J.P. was partially supported by SRC - Center for Geometry and its Applications - grant No.~2011-0030044.
L.T. acknowledge the partial support of the NSF grant DMS-1005769 and thanks P.~Zograf for useful discussions.

\section{Basic facts} \label{Basic}
Here we recall the necessary basic facts from the complex-analytic theory of \Te spaces (see the classic book \cite{Ahl66} and \cite{Ahl61,Ahl62}, and the modern exposition in \cite{IT,Nag}) and the results from \cite{TZ87a,TZ87b}.
\subsection{\Te space $T(\Gamma)$ of a Fuchsian group} \label{T space} Let $\Gamma\subset\PSL(2,\R)$ be a Fuchsian group of type $(g,n)$ acting on the Lobachevsky plane $\HH=\{z=x+\sqrt{-1}\,y\in\C \,\vert\, \im z>0\}$. The group $\Gamma$ is generated by $2g$ hyperbolic transformations $A_{1},B_{1},\dots,A_{g},B_{g}$ and $n$ parabolic transformations $S_{1},\dots,S_{n}$, where $3g-3+n>0$, satisfying the single relation
$$A_{1}B_{1}A_{1}^{-1}B_{1}^{-1}\cdots A_{g}B_{g}A_{g}^{-1}B_{g}^{-1}S_{1}\cdots S_{n}=1.$$
The group $\Gamma$ with a given, up to a conjugation in $\PSL(2,\R)$, set of generators $A_{1},B_{1},\dots,A_{g},B_{g},S_{1},\dots,S_{n}$ is called a \emph{marked} Fuchsian group.

Let $\mathcal{A}^{-1,1}(\HH,\Gamma)$ be the space of Beltrami differentials for $\Gamma$ --- a complex Banach space of $\mu\in L^{\infty}(\HH)$ satisfying
$$\mu(\gamma z)\frac{\overline{\gamma'(z)}}{\gamma'(z)}=\mu(z)\quad\forall\gamma\in\Gamma.$$
For every $\mu\in\mathcal{A}^{-1,1}(\HH,\Gamma)$ with
$$\Vert\mu\Vert_{\infty}=\sup_{z\in\HH}|\mu(z)|<1$$
there exists unique quasi-conformal (q.c.)~homeomorphism $f^{\mu}:\HH\rightarrow\HH$ satisfying the Beltrami equation
$$f_{\z}^{\mu}=\mu f^{\mu}_{z},\quad z\in\HH,$$
and fixing the points $0,1,\infty$. Then $\Gamma^{\mu}=f^{\mu}\circ\Gamma\circ(f^{\mu})^{-1}$ is a Fuchsian group of type $(g,n)$ and the \Te space $T(\Gamma)$ is defined by
$$T(\Gamma)=\{\mu\in\mathcal{A}^{-1,1}(\HH,\Gamma)\,|\,\Vert\mu\Vert_{\infty}<1\}/\sim.$$
Here $\mu\sim\nu$ if and only if $f^{\mu}\circ\gamma\circ(f^{\mu})^{-1}=f^{\nu}\circ\gamma\circ(f^{\nu})^{-1}$ for all $\gamma\in\Gamma$ (or equivalently, $f^{\mu}=f^{\nu}$ on $\R$). The group $\Ga$ corresponds to $\mu=0$ and is the origin (the base point) of $T(\Ga)$.

\subsubsection{The complex structure} \label{complex} The \Te space $T(\Gamma)$ admits a natural structure of a complex manifold, which is uniquely determined by the condition that canonical projection which sends $\mu\in\mathcal{A}^{-1,1}(\HH,\Gamma)$ with $\Vert\mu\Vert_{\infty}<1$ to its equivalence class $[\mu]\in T(\Gamma)$ is a holomorphic map. For the Fuchsian group $\Gamma$ of type $(g,n)$ the complex dimension of $T(\Gamma)$ is $d=3g-3+n$.

Explicitly this complex structure is described as follows.
Denote by $\cH^{-1,1}(\HH,\Ga)$
the finite-dimensional subspace of harmonic Beltrami differentials for $\Ga$ with respect to the
hyperbolic metric on $\HH$. It consists of $\mu\in\mathcal{A}^{-1,1}(\HH,\Ga)$ satisfying
$\pa_z(\rho \mu) = 0$, where $\rho(z)=y^{-2}$ and has complex dimension $d=3g-3+n$.
The complex vector space $\cH^{-1,1}(\HH,\Ga)$ is identified with the holomorphic
tangent space $T_{0}T(\Ga)$ to $T(\Ga)$ at the origin $\mu=0$. Every $\mu\in\cH^{-1,1}(\Ga)$
has the form $\mu(z)=y^{2}\overline{q(z)}$, where $q\in\cH^{2,0}(\HH,\Ga)$ is a
cusp form of weight 4 for $\Ga$ --- a holomorphic function on $\HH$ that vanishes at the cusps of $\Gamma$ and satisfies
\begin{equation*}
q(\ga z)\ga'(z)^2=q(z)\quad\forall\ga\in\Ga.
\end{equation*}
Correspondingly, the holomorphic cotangent space $T_{0}^{\ast}T(\Ga)$ to $T(\Ga)$ at the origin is
naturally identified with the complex vector space $\cH^{2,0}(\HH,\Ga)$,
and the pairing between $T^{\ast}_{0}T(\Ga)$ and $T_{0}T(\Ga)$ is given by
\begin{equation*}
(q, \mu)  = \iint\limits_{\Ga \bk \HH}q(z)\mu(z)d^2z,\quad\text{where}\quad d^{2}z=dxdy.
\end{equation*}

Choose a basis $\mu_1,\dots, \mu_d$ for
$\cH^{-1,1}(\HH,\Ga)$, put $\mu = \vep_1 \mu_1 + \cdots + \vep_d \mu_d$ and for $\Vert\mu\Vert_{\infty}<1$ let
$f^\mu$ be the normalized solution of the Beltrami equation. Then the correspondence
$(\vep_1, \dots, \vep_d) \mapsto \Ga^{\mu}= f^\mu \circ \Ga \circ (f^\mu)^{-1}$
defines the complex coordinates in a neighborhood of the origin in $T(\Ga)$, called \emph{Bers
coordinates}.

There is a natural isomorphism  between the \Te spaces $T(\Ga)$ and
$T(\Ga^{\mu})$, which maps $\Ga^{\nu} \in T(\Ga)$ to $(\Ga^{\mu})^{\lambda} \in
T(\Ga^{\mu})$, where, in accordance with $f^\nu=f^\lambda \circ f^\mu$,
\[
\lambda = \left(\frac{\nu - \mu}{ 1- \nu \bar{\mu}}
\frac{f_z^{\mu}}{\bar{f}_{\z}^{\mu}} \right)\circ (f^{\mu})^{-1}.
\]
This isomorphism allows to identify the holomorphic tangent space $T_{[\mu]}T(\Ga)$
at $[\mu]\in T(\Ga)$ with the complex vector space $\cH^{-1,1}(\HH,\Ga^{\mu})$, and
the holomorphic cotangent space $T^{\ast}_{[\mu]}T(\Ga)$ --- with the complex vector space
$\cH^{2,0}(\HH,\Ga^{\mu})$. It also allows to introduce the Bers coordinates in the
neighborhood of $\Ga^\mu$ in $T(\Ga)$, and to prove that these coordinates
transform complex-analytically.
\begin{remark} \label{base}
A \emph{marked} Riemann surface of type $(g,n)$ is a Riemann surface with a set of standard generators of its fundamental group, defined up to an inner automorphism.
Whence the \Te space $T(\Gamma)$ can be interpreted as a \Te space of marked Riemann surfaces of type $(g,n)$ by assigning to each  $[\mu]\in T(\Gamma)$ a marked surface $X^{\mu}\cong\Gamma^{\mu}\bk\HH$, with the surface $X\cong\Gamma\bk\HH$ playing the role of a base point. According to the isomorphism
$T(\Ga)\simeq T(\Ga^{\mu})$, the choice of a base point is inessential and we will often use the notation $T_{g,n}$ for $T(\Gamma)$.
\end{remark}

Variation formulas of the hyperbolic metric $\rho(z)|dz|^{2}$ on $\HH$ play an important role in the complex-analytic theory of \Te spaces. Put
$$(f^{\mu})^{\ast}(\rho)=\frac{|f^{\mu}_z|^{2}}{(\im f^{\mu})^{2}}.$$
The first formula is the classic result of Ahlfors \cite{Ahl61} (the so-called Ahlfors lemma)
that for $\mu\in\cH^{-1,1}(\HH,\Ga)$
\begin{equation}\label{A1-formula}
\left.\frac{\del}{\del\vep}\right|_{\vep=0}(f^{\vep\mu})^{\ast}(\rho)=0.
\end{equation}
The formula for the second variation
\begin{equation}\label{W1-formula}
\left.\frac{\del^{2}}{\del\vep_{1}\del\bar{\vep}_{2}}\right|_{\vep=0}(f^{\vep_{1}\mu+\vep_{2}\nu})^{\ast}(\rho)=\frac{1}{2}\rho f_{\mu\bar\nu},\end{equation}
where $\mu,\nu\in\cH^{-1,1}(\HH,\Ga)$ and $\Ga$-automorphic function $f_{\mu\bar\nu}$ is uniquely determined by
\begin{equation}\label{W2-formula}
-y^{2}\frac{\del^{2}f_{\mu\bar\nu}}{\del z\del\z}+\frac{1}{2}f_{\mu\bar\nu}=\mu\bar{\nu}\quad\text{and}\quad \iint\limits_{\Ga\bk\HH}|f_{\mu\bar\nu}(z)|^{2}\rho(z)d^{2}z<\infty,
\end{equation}
 was proved by S. Wolpert \cite[Theorem 3.3]{W86}.
 \begin{remark}
 It is shown in \cite[Proposition 6.3]{TT06}, that formulas \eqref{W1-formula}--\eqref{W2-formula} can be obtained from Ahlfors' earlier result in \cite{Ahl62}.
 \end{remark}

\subsubsection{\Ka metrics on $T(\Gamma)$} \label{Ka metrics}
The cotangent spaces $T^{\ast}_{[\mu]}T(\Ga)=\cH^{2,0}(\HH, \Gamma^{\mu})$ carry a natural inner product --- the Petersson's inner product on the space of cusp forms of weight $4$. It determines the Weil-Petersson metric on the \Te space $T(\Ga)$ by the formula
$$\langle \mu_{1},\mu_{2}\rangle_{\mathrm{WP}}=\iint\limits_{\Ga^{\mu}\bk\HH}\mu_{1}(z)\overline{\mu_{2}(z)}\rho(z)d^{2}z,\quad\mu_{1},\mu_{2}\in T_{[\mu]}T(\Ga)=\cH^{-1,1}(\HH,\Ga^{\mu}).$$
The Weil-Petersson metric is real-analytic and K\"{a}hler and is invariant with respect to the \Te modular group $\mathrm{Mod}(\Ga)$.

In case when $\Ga$ is a Fuchsian group of type $(g,n)$ and $n>0$,  a new \Ka metric on $T(\Ga)$ was introduced in \cite{TZ88,TZ91}. Namely, let
by $z_{1},\dots,z_{n}\in\R\cup\{\infty\}$ be the set of non-equivalent cusps for $\Ga$ --- the fixed points of the parabolic generators $S_{1},\dots,S_{n}$.
For each $i=1,\dots,n$ denote by $\Gamma_{i}$ the cyclic subgroup $\langle S_{i}\rangle $ and let $\si_{i}\in\PSL(2,\R)$ be such that $\si_{i}\infty=z_{i}$ and $\si_{i}^{-1}S_{i}\si_{i}=\left(\begin{smallmatrix}1 & \pm 1\\0 & \;\;1\end{smallmatrix}\right)$. Let $E_{i}(z,s)$ be the Eisenstein-Maass series associated with the cusp $z_{i}$, which for $\re s>1$  is defined by the following absolutely convergent series
$$E_{i}(z,s)=\sum_{\ga\in\Ga_{i}\bk\Ga}\im(\si_{i}^{-1}\ga z)^{s}.$$
The inner product
$$\langle\mu_{1},\mu_{2}\rangle_{i}=\iint\limits_{\Ga\bk\HH}\mu_{1}(z)\overline{\mu_{2}(z)}E_i(z,2)\rho(z)d^{2}z,\quad i=1,\dots,n,$$
in $\cH^{-1,1}(\HH,\Ga)$, and the corresponding inner products in all $\cH^{-1,1}(\HH,\Ga^{\mu})$ determine another Hermitian metric on $T(\Ga)$. It was proved in \cite{TZ88,TZ91} that this metric is K\"{a}hler for each $i=1,\dots,n$. In \cite{Obitsu99,Weng01, W07} it was called TZ metric and we will denote it by  $\langle~,~\rangle_{\mathrm{TZ},i}$. The metric
$\langle~,~\rangle_{\mathrm{TZ}}=\langle~,~\rangle_{\mathrm{TZ},1}+\cdots+\langle~,~\rangle_{\mathrm{TZ},n}$ is invariant with respect to the \Te modular group $\mathrm{Mod}(\Ga)$. Denote by $\omega_{\mathrm{TZ},i}$ the symplectic form of $i$-th TZ metric,
$$\omega_{\mathrm{TZ},i}=\frac{\sqrt{-1}}{2}\sum_{j,k=1}^{d}\langle\mu_{j},\mu_{k}\rangle_{\mathrm{TZ},i}d\vep_{j}\wedge d\bar\vep_{k},$$
and put $\omega_{\mathrm{TZ}}=$$\omega_{\mathrm{TZ},1}+\cdots+$$\omega_{\mathrm{TZ},n}$.

The TZ metric is intrinsically related to the second variation of the hyperbolic metric on $\HH$ (see Sect.~\ref{complex}). Namely, the following result was proved in \cite{TZ91}
\begin{equation} \label{t-z}
\lim_{y\rightarrow \infty}\im(\si_{i}z)f_{\mu\bar{\nu}}(\si_{i}z)=\frac{4}{3}\langle\mu,\nu\rangle_{\mathrm{TZ},i},\quad i=1,\dots,n,
\end{equation}
where $\mu,\nu\in\cH^{-1,1}(\HH,\Ga)$ and $f_{\mu\bar\nu}$ is defined in \eqref{W2-formula}.

\subsection{The moduli space $\cM_{0,n}$} \label{genus 0} Here
we consider the moduli space\footnote{In \cite{TZ87a, Z89} this moduli space was denoted by $W_{n}$.} $\cM_{0,n}$ of Riemann surfaces of type $(0,n)$ with labeled punctures ($n$-pointed rational curves). Each such surface is uniquely realized as $\bar{\C}=\C\cup\{\infty\}$ with $n$ labeled punctures such that the last three of them are, respectively, $0,1$ and $\infty$.
Let $\cF_{n}(\C)$ be the configuration space of $n$ ordered distinct points in $\C$ with the $\PSL(2,\C)$-action.  The moduli space is defined by $\cM_{0,n}= \cF_{n}(\C)/\PSL(2,\C)$ and is realized as the following domain in $\C^{n-3}$,
$$\cM_{0,n}=\{(w_{1},\dots,w_{n-3})\in\C^{n-3}\,|\, w_{i}\neq 0,1\;\;\text{and}\;\;w_{i}\neq w_{k}\;\;\text{for}\;\;i\neq k\}.$$

Let $X=\C\setminus\{w_{1},\dots,w_{n-3},0,1\}$ be a Riemann surface of type $(0,n)$. By the uniformization theorem, $X\cong\Gamma\bk\HH$, where type $(0,n)$ Fuchsian group $\Gamma$ is normalized such that the fixed points of $S_{n-2}$, $S_{n-1}$, $S_{n}$ are, respectively, $z_{n-2}=0$, $z_{n-1}=1$, $z_{n}=\infty$. Denote by $\HH^{\ast}$ the union of $\HH$ and all cusps for $\Ga$. There is unique covering map $J:\HH\rightarrow X$ with the group of deck transformations $\Gamma$, which extends to a holomorphic isomorphism $J:\Gamma\bk\HH^{\ast}\stackrel{\sim}{\rightarrow}\bar\C$ that fixes $0,1,\infty$ and has the property that $w_{i}=J(z_{i})$, $i=1,\dots,n-3$.
In the classical terminology $J$ is called \emph{Klein's Hauptmodul}.
It is a unique  $\Ga$-automorphic function on $\HH$ that fixes $0$ and $1$ and has a simple pole at $\infty$. The function $J$ is univalent in any fundamental domain for $\Gamma$ and has the following Fourier expansions at the cusps,
\begin{align}
J(\si_{i}z)&=w_{i}+\sum_{k=1}^{\infty}a_{i}(k)q^{k},\quad i=1,\dots,n-1,  \label{J-1} \\
J(\si_{n}z)& =\sum_{k=-1}^{\infty}a_{n}(k)q^{k},\quad  i=n, \label{J-n}
\end{align}
where $q=e^{2\pi\sqrt{-1}z}$. The first Fourier coefficients of $J$ determine the following smooth positive functions  on $\cM_{0,n}$: $h_{i}=|a_{i}(1)|^{2}$, $i=1,\dots,n-1$, and $h_{n}=|a_{n}(-1)|^{2}$.

The symmetric group $\mathrm{Symm}(n)$ acts on $\cM_{0,n}$ (see \cite[\S1]{Z89}) and let $\frak{M}_{0,n}=\cM_{0,n}/\mathrm{Symm}(n)$ be the moduli of Riemann surfaces of type $(0,n)$.
As in \cite{Z89}, let $\{f_{\sigma}\}_{\sigma\in \mathrm{Symm}(n)}$ be a $1$-cocycle for $\mathrm{Symm}(n)$ on $\cM_{0,n}$ defined by
$$f_{\sigma_{kn}}(w_1,\dots,w_{n-3})=\begin{cases}\displaystyle{\prod_{\substack{i=1\\ i\neq k}}^{n-3}}\dfrac{(w_i-w_k)^2}{w_k(w_k-1)},&  k=1,\dots, n-3,\\
\displaystyle{\prod_{i=1}^{n-3}w_{i}^2}, & k=n-2,\\
\displaystyle{\prod_{i=1}^{n-3}(w_i-1)^2}, & k=n-1
\end{cases}$$
where $\sigma_{kn}$ is the transposition interchanging the points with indices $k$ and $n$,
and extended to the full group by $f_{\sigma_{1}\sigma_{2}}=(f_{\sigma_{1}}\circ \sigma_{2})f_{\sigma_{2}}$.
Let $\lambda_{0,n}$ be the holomorphic line bundle over $\frak{M}_{0,n}$ determined by the 1-cocycle $f$, the quotient
of the trivial line bundle $\cM_{0,n}\times\C\rightarrow\cM_{0,n}$ by the symmetric group action
$$(\bm{w},z)\mapsto (\sigma\cdot\bm{w},f_\sigma(\bm{w})z),\quad\bm{w}\in\cM_{0,n}, \,z\in\C, \,\sigma\in \mathrm{Symm}(n).$$
\begin{lemma} \label{Hermite metric} A positive function $H=h_{1}\cdots h_{n-1}/h_{n}$ on $\cM_{0,n}$ determine a Hermitian metric  in the line bundle $\lambda_{0,n}$ over $\frak{M}_{0,n}$.
\end{lemma}
\begin{proof} It readily follows from the description of the symmetric group action on $\cM_{0,n}$ in \cite[\S1]{Z89} that $H(\sigma\cdot\bm{w}) |f_\sigma(\bm{w})|^2=H(\bm{w})$ for all $\bm{w}\in\cM_{0,n}$ and $\sigma\in \mathrm{Symm}(n)$.
\end{proof}

The hyperbolic metric $e^{\varphi(w)}|dw|^2$ on $X$ is a push-forward by the map $J$ of the hyperbolic metric $\rho(z)|dz|^2$ on $\HH$,
\begin{equation} \label{phi}
e^{\varphi(w)}=\frac{\left|(J^{-1})'(w)\right|^2}{\left(\text{Im}\,J^{-1}(w)\right)^2},
\end{equation}
and satisfies the Liouville equation
\begin{equation} \label{Liov}
\varphi_{w\bar{w}}=\frac{1}{2}e^{\varphi},\quad w\in X.
\end{equation}
From the Fourier expansions \eqref{J-1}--\eqref{J-n} one gets the following asymptotic behavior of $\varphi(w)$
as $w\rightarrow w_i$ (see \cite[Lemma 2]{TZ87a})
\begin{align}
\varphi(w) & =-2\log |w-w_i|-2\log \left|\log \left|\frac{w-w_i}{a_i(1)}\right|\right|+O(|w-w_i|),\ \ i\neq n,  \label{as-1}\\
\varphi(w)&=-2\log\left|w\right|-2\log\log \left|\frac{w}{a_n(-1)}\right|+O(|w|^{-1}),\ \  i=n. \label{as-2}
\end{align}
Denote by $\mathcal{S}(f)$ the Schwarzian derivative,
$$\mathcal{S}(f)=\frac{f'''}{f'}-\frac{3}{2}\left(\frac{f''}{f'}\right)^{2}.$$
We have
\begin{equation} \label{Schwarz}
\mathcal{S}(J^{-1})(w)=\varphi_{ww}(w)-\frac{1}{2}\varphi_{w}(w)^{2}=\sum_{i=1}^{n-1}\left(\frac{1}{2(w-w_{i})^{2}}+\frac{c_{i}}{w-w_{i}}\right),
\end{equation}
and
$$\quad\mathcal{S}(J^{-1})(w)=\frac{1}{2w^{2}}+O(|w|^{-3})\quad\text{as}\quad w\rightarrow\infty$$
where $c_{i}=-a_{i}(2)/(a_{i}(1))^{2}$, $i=1,\dots,n-1$, (see \cite[Lemma 1]{TZ87a}) are \emph{accessory parameters} of the Fuchsian uniformization of the surface $X$.

Consider the Riemann surface $X=\C\setminus\{w_{1},\dots,w_{n-3},0,1\}\cong\Gamma\bk\HH$ as a base point in $T_{0,n}$.
For each $[\mu]\in T_{0,n}$ the Fuchsian group $\Gamma^{\mu}=f^{\mu}\circ\Gamma\circ (f^{\mu})^{-1}$ is normalized and we realize the Riemann surface $X^{\mu}\cong\Gamma^{\mu}\bk\HH$ as $X^{\mu}=\bar{\C}\setminus\{w_{1}^{\mu},\dots,w_{n-3}^{\mu}, 0,1,\infty\}$. Denote by $J_{\mu}$ the corresponding normalized covering map $J_{\mu}:\HH\rightarrow X^{\mu}$ and consider the map $p: T_{0,n}\rightarrow\cM_{0,n}$, defined by
$$T_{0,n}\ni[\mu] \mapsto p([\mu])=(w_{1}^{\mu},\dots,w_{n-3}^{\mu})\in\cM_{0,n},\quad\text{where}\quad w_{i}^{\mu}=(J_{\mu}\circ f^{\mu})(z_{i}).$$

According to \cite[Lemma 3]{TZ87a}, the map $p$
is a complex-analytic covering.
Consider the commutative diagram
\begin{align}\label{CD-0}
\begin{CD}
\mathbb{H} @> f^{\mu}>> \HH   \\
@VV J V   @VV J_{\mu}  V\\
X @>   F^{\mu} >> X^{\mu}
\end{CD}
\end{align}
It follows from \eqref{CD-0} that
\begin{equation} \label{F}
F^{\mu}_{\bar{w}}=MF^{\mu}_{w},\quad\text{where}\quad M=(\mu\circ J^{-1})\frac{\overline{(J^{-1})'}}{(J^{-1})'}.
\end{equation}
The function $f^{\vep\mu}(z)$ is real-analytic in $\vep$ for all $z\in\C$. Put $\dot{f}^{\mu}(z)=\del f^{\vep\mu}/\del\vep|_{\vep=0}$.
It satisfies $\dot{f}^{\mu}_{\bar{z}}=\mu$
and is given by
$$\dot{f}^{\mu}(z)=-\frac{1}{\pi}\iint_{\mathbb{H}}\mu(\zeta)R(\zeta,z)d^{2}\zeta,\quad\text{where}\quad R(\zeta,z)=\frac{z(z-1)}{(\zeta-z)\zeta(\zeta-1)}.$$
Correspondingly, the function $F^{\vep\mu}$ is holomorphic in $\vep$ and
\begin{equation} \label{dot F}
\dot{F}^{\mu}(w)=-\frac{1}{\pi}\iint_{\C}M(v)R(v,w)d^{2}v,
\end{equation}
where $\dot{F}^{\mu}=\left.\left(\del F^{\vep\mu}/\del\vep\right)\right|_{\vep=0}$.

Denote by $r_{i}$, $i=1,\dots,n-3$, the basis in $\cH^{2,0}(\HH, \Gamma)=T^{\ast}_{0}T_{0,n}$ defined as
$$r_{i}(z)=R_{i}(J(z))J'(z)^{2},\quad\text{where}\quad R_{i}(w)=-\frac{1}{\pi}R(w,w_{i}),$$
and let $q_{i}(z)$, $i=1,\dots,n-3$, be the basis in $\cH^{2,0}(\HH, \Gamma)$, biorthogonal to $r_{i}(z)$ with respect to the Petersson inner product. Finally, let $\mu_{i}(z)=y^{2}\overline{q_{i}(z)}$ be the corresponding basis in $\cH^{-1,1}(\HH, \Gamma)=T_{0}T_{0,n}$ (see \cite[Sects.~2.4--2.6]{TZ87a}).
The bases $r^{\mu}_{i}(z)$ in
$\cH^{2,0}(\HH, \Gamma^{\mu})$ and $\mu_{i}^{\mu}(z)$ in $\cH^{-1,1}(\HH, \Gamma^{\mu})$ for $[\mu]\in T_{0,n}$ are defined similarly. Then for the covering map $p: T_{0,n}\rightarrow\cM_{0,n}$ we have (see \cite[Lemma 3]{TZ87a})
\begin{equation} \label{motion}
dp_{[\mu]}\left(\mu_{i}^{\mu}\right)=\frac{\del}{\del w_{i}} \quad\text{and}\quad p^{\ast}_{[\mu]}(dw_{i})=r^{\mu}_{i},\quad i=1,\dots,n-3.
\end{equation}
We also have
\begin{equation} \label{Schwarz-2}
\mathcal{S}(J^{-1})(w)=\sum_{i=1}^{n}\mathcal{E}_{i}(w)-\pi\sum_{i=1}^{n-3}c_{i}R_{i}(w),
\end{equation}
where
\begin{equation} \label{E-series-0}
\mathcal{E}_{i}(w)=\frac{1}{2(w-w_{i})^{2}}-\frac{1}{2w(w-1)},\;\; i\neq n,\quad\mathcal{E}_{n}(w)=\frac{1}{2w(w-1)}.
\end{equation}
The corresponding functions $e_{i}(z)=\mathcal{E}_{i}(J(z))J'(z)^{2}$ on $\HH$ are automorphic forms of weight $4$ for $\Ga$ with non-zero constant term at the cusp
$z_{i}$, $i=1,\dots,n$.

Put $\dot{F}^{i}=\dot{F}^{\mu_{i}}$, $i=1,\dots,n-3$. The functions $\dot{F}^{i}(w)$ are given by \eqref{dot F} where
\begin{equation} \label{M}
M_{i}(v)=e^{-\varphi(v)}\overline{Q_{i}(v)},\quad Q_{i}(v)=q_{i}(J^{-1}(v))(J^{-1})^{\prime}(v)^{2}.
\end{equation}
It follows from \eqref{motion} that
$$\dot{F}^{i}(w_{j})=\delta_{ij},\quad i,j=1,\dots,n-3$$
and
$$\dot{F}^{i}(0)=\dot{F}^{i}(1)=0,\quad \dot{F}^{i}(w)=o(|w|^{2})\quad\text{as}\quad w\rightarrow\infty.$$
Also we have
\begin{align}
\dot{F}^{i}(w) & =\delta_{ij}+(w-w_{j})\dot{F}_{w}^{i}(w)+o\left(\left|\frac{w-w_{j}}{\log|w-w_{j}|}\right|\right) \label{F-limits-1}\\
\intertext{as $w\rightarrow w_{j}$, $ j\neq n$, and}
\dot{F}^{i}(w) & =w \dot{F}^{i}_{w}(w) + o\left(\frac{|w|}{\log |w|}\right) \quad \text{as} \ \ w\rightarrow\infty. \label{F-limits-2}
\end{align}

\begin{remark} One can easily prove \eqref{F-limits-1}--\eqref{F-limits-2} (with better error terms) using integral representation \eqref{dot F} and asymptotic behavior \eqref{as-1}--\eqref{as-2}.
Here is the sketch of the proof of \eqref{F-limits-2}. We have
$$\dot{F}^{i}(w)-w\dot{F}^{i}_{w}(w)=-\frac{1}{\pi}\iint_{\C}M_{i}(v)\left\{\frac{v-2w}{(v-w)^{2}}-\frac{1}{v}\right\}d^{2}v,$$
where the integral is understood in the principal value sense as in \cite{Ahl66}. Putting $v=uw$ we have
 $$\dot{F}^{i}(w)-w\dot{F}^{i}_{w}(w)=-\frac{\bar{w}}{\pi}\iint_{\C}M_{i}(uw)\left\{\frac{u-2}{(u-1)^{2}}-\frac{1}{u}\right\}d^{2}u.$$
 It follows from \eqref{M} that
 $$M_{i}(w)=O\left(\frac{\log^{2}|w|}{|w|}\right)\quad\text{as}\quad w\rightarrow\infty,$$
whence the integral over $|u|\geq 2$ is estimated by $O(\log^{2}|w|)$. Now choose $\alpha^{3}=\log^{2}|w|/|w|$. Estimating the integral over $|u|\leq \alpha$ by the area, and the integral over $\alpha\leq |u|\leq 2$ --- by the estimate of $M_{i}(\alpha|w|)$, we obtain that both of these integrals are estimated by $|w|\alpha^{2}=|w|^{1/3}\log^{4/3}|w|$.
\end{remark}

Let $e^{\varphi^{\mu}(w)}|dw|^{2}$ be  the hyperbolic metric on the Riemann  surface $X^{\mu}$.
It follows from commutative diagram \eqref{CD-0} that  $F^{\mu}\circ J=J_{\mu}\circ f^{\mu}$, and we have
$$(F^{\mu})^{\ast}(e^{\varphi^{\mu}})=(J^{-1})^{\ast}(f^{\mu})^{\ast}(\rho).$$
Whence the first and second variations of the family of hyperbolic metrics
$e^{\varphi^{\vep\mu}(w)}|dw|^{2}$ on the Riemann surfaces $X^{\vep\mu}=F^{\vep\mu}(X)$ are given by the same formulas \eqref{A1-formula}--\eqref{W2-formula}, where $\rho$ is replaced by $e^{\varphi}$ and $f_{\mu\bar\nu}$ --- by $(J^{-1})^{\ast}(f_{\mu\bar\nu})=f_{\mu\bar\nu}\circ J^{-1}$. Moreover, since for any $\alpha\in\R$ we have
$$(F^{\mu})^{\ast}(e^{\alpha\varphi^{\mu}})=\left((F^{\mu})^{\ast}(e^{\varphi^{\mu}})\right)^{\alpha},$$
we get from \eqref{A1-formula} and \eqref{W1-formula}
\begin{equation} \label{A2-formula}
\left.\frac{\del}{\del\vep}\right|_{\vep=0}(F^{\vep\mu})^{\ast}(e^{\alpha\varphi^{\vep\mu}})=0
\end{equation}
and
\begin{equation} \label{W3-formula}
\left.\frac{\del^{2}}{\del\vep_{1}\del\bar{\vep}_{2}}\right|_{\vep=0}(F^{\vep_{1}\mu+\vep_{2}\nu})^{\ast}(e^{\alpha\varphi^{\vep\mu}})=\frac{\alpha}{2}e^{\alpha\varphi} f_{\mu\bar\nu}\circ J^{-1}.
\end{equation}

Finally, each TZ metric $\langle~,~\rangle_{\mathrm{TZ},i}$ on $T_{0,n}$ is invariant with respect to the automorphism group of the covering
$p:T_{0,n}\rightarrow\cM_{0,n}$ and determines a \Ka metric on $\cM_{0,n}$, which we continue to denote by $\langle~,~\rangle_{\mathrm{TZ},i}$, $i=1,\dots,n$.

\subsection{The Schottky space $\mathfrak{S}_{g,n}$} \label{S space}
A Schottky group $\Sigma$ is a free finitely generated strictly loxodromic Kleinian group. Its limit set $\Lambda$ is a Cantor set and the region of discontinuity $\Omega=\bar\C\setminus\Lambda$ is connected. Let $\Sigma$ be a Schottky group of rank $g>1$, considered as a discrete subgroup of $\mathrm{PSL}(2,\C)$. The group $\Sigma$ acts on $\Omega$ freely, and the quotient space $\Sigma\bk\Omega$ is compact Riemann surface of genus $g$. A Schottky group $\Sigma$ of rank $g$ with a relation-free system of generators $L_{1},\dots,L_{g}$ is called \emph{marked}. For each such system of free generators there is a fundamental domain $D$ for $\Sigma$ in $\Omega$ which is a region in $\bar\C$ bounded by $2g$ disjoint Jordan curves $C_{1},\dots,C_{g},C'_{1},\dots,C'_{g}$ with $C'_{i}=-L_{i}(C_{i})$, $i=1,\dots,g$. Here $C_{i}$ and $C_{i}'$ are oriented as components of the boundary of $D$, and the minus sign means the
reverse orientation. Each element $L_{i}$ can be represented in the normal form
$$\frac{L_{i}w-a_{i}}{L_{i}w-b_{i}}=\lambda_{i}\frac{w-a_{i}}{w-b_{i}},\quad w\in\bar\C,$$
where $a_{i}$ and $b_{i}$ are the respective attracting and repelling fixed points of the transformation
$L_{i}$ and $0 <|\lambda_{i}| < 1$. In what follows we always assume that a marked Schottky is \emph{normalized}, that is $a_{1}=0$, $b_{1}=\infty$ and $a_{2} = 1$. In particular, this implies that  $\infty\notin D$.
The mapping
$$(\Sigma; L_{1},\dots,L_{g})\mapsto (a_{3}, \dots, a_{g}, b_{2},\dots,b_{g},\lambda_{1},\dots,\lambda_{g})\in\C^{3g-3}$$
establishes an one-to-one correspondence between the set of normalized marked Schottky
groups and a region $\mathfrak{S}_{g}$ in $\C^{3g-3}$, called the \emph{Schottky space}.

Equivalently, the Schottky space is defined as follows. Let $\mathcal{A}^{-1,1}(\Omega, \Sigma)$ be the complex Banach space of $L^{\infty}(\Omega)$ of Beltrami differentials for $\Sigma$  (cf.~Sect.~\ref{T space}).  Let $\mathfrak{D}(\Sigma)$ be a deformation space of the Schottky group $\Sigma$,
$$\mathfrak{D}(\Sigma)=\{\mu\in\mathcal{A}^{-1,1}(\Omega, \Sigma)\,|\,\Vert\mu\Vert_{\infty}<1\}/\sim,$$
where $\mu\sim\nu$ if and only if $F^{\mu}\circ\sigma\circ(F^{\mu})^{-1}=F^{\nu}\circ\sigma\circ(F^{\nu})^{-1}$ for all $\sigma\in\Sigma$ (or equivalently, $F^{\mu}=F^{\nu}$ on $\Lambda$)\footnote{Here and in what follows $F^{\mu}$ is a normalized solution of the Beltrami equation on $\C$ with Beltrami coefficient $\mu$.}.
The group $\Sigma$ corresponds to $\mu=0$ and is the origin (the base point) of $\mathfrak{D}(\Sigma)$. The deformation space $\mathfrak{D}(\Sigma)$ is complex-analytically isomorphic to the Schottky space $\mathfrak{S}_{g}$ with the choice of a base point.

The Schottky space $\mathfrak{S}_{g,n}$ of type $(g,n)$ Riemann surfaces is defined by a holomorphic fibration $\jmath: \mathfrak{S}_{g,n}\rightarrow\mathfrak{S}_{g}$ whose fibers over the points $[\mu]\in\mathfrak{S}_{g}$ are configuration spaces $\cF_{n}(\Sigma^{\mu}\bk\Omega^{\mu})$, where $\Sigma^{\mu}=F^{\mu}\circ\Sigma\circ(F^{\mu})^{-1}$ and $\Omega^{\mu}=F^{\mu}(\Omega)$. Equivalently it is defined as follows. Consider the deformation space of a Schottky group $\Sigma$ together with a point $(w_{1},\dots,w_{n})\in\cF_{n}(D)$,
\begin{align*}
\mathfrak{D}&(\Sigma; w_{1},\dots,w_{n})\\
&=\{(\mu;w^{\mu}_{1},\dots,w^{\mu}_{n})\in\mathcal{A}^{-1,1}(\Omega, \Sigma)\times \cF_{n}(D^{\mu})\,|\,\Vert\mu\Vert_{\infty}<1\}/\sim.
\end{align*}
Here $w^{\mu}_{i}=F^{\mu}(w_{i})$, $D^{\mu}=F^{\mu}(D)$ and $\mu\sim\nu$ if and only if $F^{\mu}\circ\sigma\circ(F^{\mu})^{-1}=F^{\nu}\circ\sigma\circ(F^{\nu})^{-1}$ for all $\sigma\in\Sigma$ and $w^{\mu}_{i}=w^{\nu}_{i}$, $i=1,\dots,n$.
The deformation space $\mathfrak{D}(\Sigma;w_{1},\dots,w_{n})$ is complex-analytically isomorphic to the Schottky space $\mathfrak{S}_{g,n}$ with the choice of a base point.

Let $X=\Sigma\bk\Omega$ be compact Riemann surface of genus $g$ with $n$ marked points $x_{1},\dots,x_{n}$, and let $\Gamma$ be a Fuchsian group of type $(g,n)$ such that $X_{0}=X\setminus\{x_{1},\dots,x_{n}\}\cong\Gamma\bk\HH$. One can choose generators $A_{1},B_{1},\dots,A_{g},B_{g}$ and $S_{1},\dots,S_{n}$ of $\Gamma$ such that $\Sigma$ is isomorphic to the
quotient group $\Gamma/N$, where $N$ is the smallest normal subgroup of $\Gamma$ which contains $A_{1},\dots,A_{g}$ and $S_{1},\dots,S_{n}$. As in Sect.~\ref{genus 0}, let $\HH^{\ast}$ be the union of $\HH$ and all cusps for $\Gamma$. The complex-analytic covering
$\pi_{\Gamma}:\HH\rightarrow\Gamma\bk\HH\cong X_{0}$ extends to the map $\pi_{\Gamma}^{\ast}:\HH^{\ast}\rightarrow X$ such that $\pi_{\Gamma}^{\ast}(z_{i})=x_{i}$, where $z_{i}$ are fixed points of $S_{i}$, $i=1,\dots,n$.

The Schottky uniformization of a compact Riemann surface $X$ with $n$ marked points $x_{1},\dots,x_{n}$ is related to the Fuchsian uniformization of a punctured
surface $X_{0}=X\setminus\{x_{1},\dots,x_{n}\}$ by the commutative diagram
\begin{equation*}
\begin{tikzcd}[thick]
\HH^{\ast} \arrow{r}{J} \arrow[swap]{dr}{\pi_{\Gamma}^{\ast}} & \Omega \arrow{d}{\pi_{\Sigma}} \\
& X
\end{tikzcd}
\end{equation*}
where $\pi_{\Sigma}$ is unramified while $J$ and $\pi_{\Ga}^{\ast}$ are branched covering maps. The map $J$ is considered as a meromorphic function on $\HH$ which is automorphic with respect to $N$ and satisfies
$J\circ B_{i}=L_{i}\circ J$, $j=1,\dots,g$. It has the following Fourier series expansions at the cusps of $\Gamma$,
\begin{equation} \label{F-J}
J(\sigma_i z)=w_i+\sum_{k=1}^{\infty} a_{i}(k) q^{k},\quad i=1,\dots,n,
\end{equation}
where $w_{i}=J(z_{i})$. (Cf.~\eqref{J-1} and note that since $\Sigma$ is normalized, $\infty\notin\Omega$.)

Let $\mathscr{L}_{i}$ be $i$-th the \emph{relative dualizing sheaf} on $\mathfrak{S}_{g,n}$ --- a holomorphic line bundle dual to the vertical tangent bundle on $\mathfrak{S}_{g,n}$ along the fibers of the projection $p_{i}:\mathfrak{S}_{g,n}\rightarrow\mathfrak{S}_{g,n-1}$ which `forget' the marked point $w_{i}$, $i=1,\dots,n$.  The bundles $\cL_{i}$, also called \emph{tautological line bundles}, are characterized by the property that the fiber of $\mathscr{L}_{i}$ over a point $([\Sigma];w_{1},\dots,w_{n})\in\mathfrak{S}_{g,n}$ is the cotangent line $T^{\ast}_{w_{i}}(\Sigma\backslash\Omega)$.

Since $J\circ B_{k}=L_{k}\circ J$, the points $w_{1},\dots,L_{k}w_{i},\dots,w_{n}$ correspond to the cusps $z_{1},\dots,B_{k}z_{i},\dots,z_{n}$, and the first Fourier coefficient of $J(z)$ at the equivalent cusp $B_{k}z_{i}$ is $L'_{k}(w_{i})a_{i}(1)$.  Correspondingly,  $h_{i}=|a_{i}(1)|^{2}$ gets replaced by $h_{i}|L'_{k}(w_{i})|^{2}$, and using the above interpretation of the line bundles $\cL_{i}$ we arrive at the following statement.
\begin{lemma}\label{h tautological} The quantities $h_{i}$ determine Hermitian metrics in the holomorphic line bundles $\cL_{i}$, $i=1,\dots,n$.\end{lemma}

Let $e^{\varphi(w)}|dw|^{2}$ be the push-forward of the hyperbolic metric on $\HH$ by the map $J$. It is given by the same formula \eqref{phi}, where $\varphi(w)$  is smooth on  $\Omega_{0}=\Omega\setminus\Sigma\cdot\{w_{1},\dots,w_{n}\}$, a complement in $\Omega$ of the $\Sigma$-orbit of $\{w_{1},\dots,w_{n}\}$. The function $\varphi(w)$ satisfies
\begin{equation} \label{phi-transform}
\varphi(\si w)=\varphi(w)-\log|\si'(w)|^{2},\quad\forall\si\in\Sigma,\;w\in\Omega_{0},
\end{equation}
and has the same asymptotics \eqref{as-1} as $w\rightarrow w_{i}$, $i=1,\dots,n$.
\begin{remark} \label{h-formula} From asymptotics \eqref{as-1} it follows that
$$\log h_{i}  =\lim_{w\rightarrow w_{i}}\left(\log|w-w_{i}|^{2}+\frac{2e^{-\varphi(w)/2}}{|w-w_{i}|}\right),\quad i=1,\dots,n.$$
\end{remark}
To each marked Fuchsian group $\Ga$ of type $(g,n)$ there is a unique marked normalized Schottky group $\Sigma\simeq\Gamma/N$ with the domain of discontinuity $\Omega$ such that $\Ga\bk\HH^{\ast}\cong\Sigma\bk\Omega$. This determines a map
$$\pi: T_{g,n}\rightarrow\mathfrak{S}_{g,n}$$
by putting $w_{i}=J(z_{i})$, $i=1,\dots,n$. As in $n=0$ case (see \cite[Sect. 2.4]{TZ87b}), the map $\pi$  is a complex-analytic covering. It plays the same role as the corresponding covering map $p$ in Sect.~\ref{genus 0}.

Specifically, the push-forward by the map $J$ of the vector space $\cH^{2,0}(\HH,\Gamma)$ is a vector space $\cH^{2,0}(\Omega_{0},\Sigma)$ of holomorphic functions on $\Omega_{0}$, defined as
$$Q(w)=q(J^{-1}(w))(J^{-1})'(w)^{2},\quad q(z)\in\cH^{2,0}(\HH,\Gamma).$$
They are automorphic forms of weight $4$ for group $\Sigma$ which admit a meromorphic extension to $\Omega$ with at most simple poles at $\Sigma\cdot\{w_{1},\dots,w_{n}\}$.
The space  $\cH^{2,0}(\Omega_{0},\Sigma)$ is naturally identified with the holomorphic cotangent space $T^{\ast}_{0}\mathfrak{S}_{g,n}$ to $\mathfrak{S}_{g,n}$ at the origin.
Correspondingly, the holomorphic tangent space $T_{0}\mathfrak{S}_{g,n}$ is the complex vector space $\cH^{-1,1}(\Omega_{0},\Sigma)$ of Beltrami differentials, harmonic with respect to the hyperbolic metric on $\Omega_{0}$. Namely, each $M\in\cH^{-1,1}(\Omega_{0},\Sigma)$ has the form
$$M(w)=e^{-\varphi(w)}\overline{Q(w)},\quad Q\in\cH^{2,0}(\Omega_{0},\Sigma).$$
The tangent and cotangent spaces to $\mathfrak{S}_{g,n}$ at each point $(\Sigma^{\mu},w_{1}^{\mu},\dots,w_{n}^{\mu})$ are identified, respectively, with $\cH^{-1,1}(\Omega_{0}^{\mu},\Sigma^{\mu})$ and $\cH^{2,0}(\Omega_{0}^{\mu},\Sigma^{\mu})$, where $\Omega^{\mu}_{0}=F^{\mu}(\Omega_{0})$.
We have the following analog of commutative diagram \eqref{CD-0},
\begin{align}\label{CD-S}
\begin{CD}
\mathbb{H} @> f^{\mu}>> \HH   \\
@VV J V   @VV J_{\mu}  V\\
\Omega_{0} @>   F^{\mu} >> \Omega^{\mu}_{0}
\end{CD}
\end{align}
Here $F^{\vep\mu}$ satisfies Beltrami equation \eqref{F}, is complex-analytic in $\vep$ and $\dot{F}^{\mu}$ is given by \eqref{dot F}.

From the fibration $\jmath: \mathfrak{S}_{g,n}\rightarrow \mathfrak{S}_{g}$ it follows that  $T^{\ast}_{0}\mathfrak{S}_{g,n}$ has a subspace $\jmath^{\ast}(T^{\ast}_{0}\mathfrak{S}_{g})\cong\cH^{2,0}(\Omega,\Sigma)$ with a natural basis $P_{1}(w),\dots,P_{3g-3}(w)$ given by holomorphic automorphic forms  of weight $4$ for $\Sigma$ which represent the cotangent vectors $d\lambda_{1},\dots,d\lambda_{g},da_{3},\dots,da_{g},db_{2},\dots,db_{g}$ as in \cite[formulas (2.2)]{Z89}.
The complementary subspace to $\jmath^{\ast}(T^{\ast}_{0}\mathfrak{S}_{g})$ in  $T^{\ast}_{0}\mathfrak{S}_{g,n}$ is isomorphic to the subspace $T^{\ast}_{0}\cF_{n}(X)$, the cotangent space to the configuration space at the base point $(w_{1},\dots,w_{n})$. Its natural basis, as it follows from \eqref{dot F}, is given by the following meromorphic automorphic forms of weight $4$,
\begin{equation} \label{P-series}
P_{3g-3+i}(w)=-\frac{1}{\pi}\sum_{\si\in\Sigma}R(\si w,w_{i})\si'(w)^{2},\quad w\in\Omega,
\end{equation}
which represent $dw_{i}$, $i=1,\dots,n$.

Denote by $M_{1}(w),\dots, M_{d}(w)$ the basis in $\cH^{-1,1}(\Omega_{0},\Sigma)$, dual to the basis  $P_{1}(w),\dots,P_{d}(w)$ in $\cH^{2,0}(\Omega_{0},\Sigma)$ with respect to the pairing
\begin{equation} \label{pair}
(Q,M)=\iint_{D}Q(w)M(w)d^{2}w.
\end{equation}
Here $M_{3g-3+1},\dots,M_{3g-3+n}$ represent the tangent vectors $\del/\del w_{1},\dots,\del/\del w_{n}$ in $T_{0}\mathfrak{S}_{g,n}$. The corresponding bases in tangent and cotangent spaces to $\mathfrak{S}_{g,n}$ at arbitrary point $(\Sigma^{\mu}; w_{1}^{\mu},\dots,w_{n}^{\mu})$ are defined similarly.

As in Sect.~\ref{genus 0}, we have $\mathcal{S}(J^{-1})(w)=\varphi_{ww}(w)-\frac{1}{2}\varphi_{w}(w)^{2}$. It follows from the asymptotic behavior of $\varphi(w)$ as $w\rightarrow w_{i}$, that
\begin{equation} \label{S-J Schottky}
\mathcal{S}(J^{-1})(w)= \sum_{i=1}^{n}\cE_{i}(w) -\pi\sum_{l=1}^{3g-3+n}c_{l}P_{l}(w),
\end{equation}
where (cf.~\eqref{Schwarz-2} and \eqref{E-series-0})
\begin{equation} \label{E-series-S}
\cE_{i}(w)=\frac{1}{2}\sum_{\si\in\Sigma}\left(\frac{1}{(\si w-w_{i})^{2}}-\frac{1}{\si w(\si w-1)}\right)\si'(w)^{2},\quad i=1,\dots,n,
\end{equation}
are meromorphic automorphic forms of weight $4$ for $\Sigma$ with the second order poles at $\Sigma\cdot w_{i}$,
and $c_{1},\dots,c_{d}$ are the analogs of accessory parameters\footnote{Note that for $i=1,\dots,3g-3$ parameters $c_{i}$ introduced here are $-1/\pi$ times accessory parameters in \cite{TZ87b}. }.

For the first and second variations of the family of hyperbolic metrics on the Schottky domains $\Omega^{\mu}$ we have the same formulas \eqref{A2-formula}--\eqref{W3-formula}.
Finally,  each TZ metric $\langle~,~\rangle_{\mathrm{TZ},i}$ on $T_{g,n}$ is invariant with respect to the automorphism group of the covering
$\pi:T_{g,n}\rightarrow\mathfrak{S}_{g,n}$ and determines a \Ka metric on $\mathfrak{S}_{g,n}$, which we continue to denote by $\langle~,~\rangle_{\mathrm{TZ},i}$, $i=1,\dots,n$.

\section{Liouville action}\label{Liouv}
\subsection{Punctured spheres} \label{spheres} Let $X=\C\setminus\{w_{1},\dots,w_{n-3},0,1\}$ be a marked Riemann surface of type $(0,n)$. The regularized classical Liouville action is defined by the following formula (see \cite{TZ87a}),
\begin{align}
&S(w_{1},\dots,w_{n-3}) \nonumber\\
=&\lim_{\delta\rightarrow 0^{+}}\left(\iint\limits_{X_{\delta}}(|\varphi_{w}|^{2}+e^{\varphi})d^{2}w +2\pi n\log\delta + 4\pi(n-2)\log|\log\delta|\right), \label{L-0}
\end{align}
where $\displaystyle{X_{\delta}=\C\setminus\cup_{i=1}^{n-1}\{|w-w_{i}|<\delta\}\cup\{|w|>1/\delta\}}$. It is a critical value of the \emph{Liouville action}, the Euler-Lagrange functional for the Liouville equation \eqref{Liov} with the asymptotic behavior \eqref{as-1}--\eqref{as-2} on the Riemann surface $X$, and defines the smooth function $S:\cM_{0,n}\rightarrow\R$. Denote by $\del$ and $\delb$ the $(1,0)$ and $(0,1)$ components of de Rham differential on $\cM_{0,n}$. It is proved in  \cite[Theorem 1]{TZ87a},
$$\del S=-2\pi\sum_{i=1}^{n-3}c_{i}R_{i},$$
so that the regularized Liouville action is a generating function for the accessory parameters,
$$c_{i} = -\frac{1}{2\pi}\frac{\del S}{\del w_{i}},\quad i=1,\dots,n-3.$$
Also, according to \cite[Theorem 2]{TZ87a}, the function $-S$ is a \Ka potential for the Weil-Petersson metric on $\cM_{0,n}$,
$$\delb\del S=-2\sqrt{-1}\omega_{\mathrm{WP}}.$$

Let $\frak{M}_{0,n}=\cM_{0,n}/\mathrm{Symm}(n)$ be the moduli space of Riemann surfaces of type $(0,n)$. It is proved in \cite[\S1]{Z89} that $\exp\{S/\pi\}$ determines a Hermitian metric in a holomorphic line bundle $\lambda_{0,n}$ over $\mathfrak{M}_{0,n}$ (see Sect.~\ref{genus 0}), so that
\begin{equation}\label{Chern S}
c_{1}(\lambda_{0,n},\exp\{S/\pi\})=\frac{1}{\pi^{2}}\omega_{\mathrm{WP}}.
\end{equation}

\subsection{Schottky domains} \label{S domains} Let $\Sigma$ be a marked normalized Schottky group of rank $g>1$. The classical Liouville action is a critical value of the Liouville action functional and is defined by the following formula \cite{TZ87b} (see \cite{TT03} for the cohomological interpretation),
\begin{gather} \label{L-S}
S(\varphi) = \frac{\sqrt{-1}}{2}\iint_{D}\omega(\varphi)+\frac{\sqrt{-1}}{2}\sum_{k=2}^{g}\int_{C_{k}}\theta_{L_{k}^{-1}}(\varphi),
\end{gather}
where
$$\omega(\varphi)=(|\varphi_{w}|^{2}+e^{\varphi})dw\wedge d\bar{w}$$
and for $\si\in\mathrm{PSL}(2,\C)$
$$\theta_{\si^{-1}}(\varphi)=\left(\varphi-\frac{1}{2}\log|\si'|^{2}
-\log|c(\si)|^2\right) \left(\frac{\si''}{\si'} dw -
\frac{\ov{\si''}}{\ov{\si'}}d\bar{w}\right).$$
Here for $\si=\left(\begin{smallmatrix} a & b\\c & d\end{smallmatrix}\right)$ we put $c(\gamma)=c$, so that $\theta_{\si^{-1}}(\varphi)=0$ if $c(\si)=0$.

The classical Liouville action is independent of the choice of a fundamental domain $D$ for the marked Schottky group $\Sigma$ and determines a smooth function $S:\mathfrak{S}_{g}\rightarrow \R$. As in Sect.~\ref{spheres}, denoting by $\del$ and $\delb$ the $(1,0)$ and $(0,1)$ components of
de Rham differential on $\mathfrak{S}_{g}$ we have\footnote{See the previous footnote.} (see \cite[Theorems 1,2]{TZ87b})
$$\del S =-2\pi\sum_{l=1}^{3g-3}c_{l}P_{l}\quad\text{and}\quad \delb\del S=-2\sqrt{-1}\omega_{\mathrm{WP}},$$
so that $-S$ is a \Ka potential for the Weil-Petersson metric on $\mathfrak{S}_{g}$.

To define the classical Liouville action for the hyperbolic metric on $\Omega_{0}=\Omega\setminus\Sigma\cdot\{w_{1},\dots,w_{n}\}$ one needs to regularize the area integral in \eqref{L-S}, which diverges due to the asymptotic behavior \eqref{as-1} of $\varphi$ as $w\rightarrow w_{i}$. We do it in the same way as in genus $0$ case. Namely, suppose that
all $w_{1},\dots,w_{n}\in\mathrm{Int}D$, the interior of $D$, and for sufficiently small $\delta>0$ define $D_{\delta}=D\setminus \cup_{i=1}^{n}D_{i}(\delta)$,
where $D_{i}(\delta)=\{|w-w_{i}|<\delta\}\subset D$, $i=1,\dots,n$. It follows from \eqref{as-1} that the following limit exists
\begin{equation} \label{e:bulk}
S_{\mathrm{bulk}}(\varphi)=\lim_{\delta\rightarrow 0^{+}}\left(\frac{\sqrt{-1}}{2}\iint_{D_{\delta}}\omega(\varphi) +2\pi n(\log\delta + 2\log\left|\log\delta\right|)\right).
\end{equation}
\begin{remark} \label{curves} Equivalently, one can define $S_{\mathrm{bulk}}(\varphi)$ by cutting out the interiors $D_{i}\subset D$ of arbitrary simple closed curves $l_{i}$ around $w_{i}$ such that $w_{j}\notin D_{i}$ for $i\neq j$. Namely, let
\begin{align*}
 &\frac{2}{\sqrt{-1}}\tilde{S}_{l}(\varphi)\\
=&\iint_{D\setminus\cup_{i=1}^{n}D_{i}}\omega(\varphi)
+\sum_{i=1}^{n}\int_{l_{i}}\left(\frac{2\log|w-w_i|}{\bar{w}-\bar{w}_i}+
\frac{2\log\left(\log|w-w_i|\right)^2}{\bar{w}-\bar{w}_i}\right)d\bar{w}.
\end{align*}
Then it easily follows from Stokes' theorem and \eqref{as-1} that
$$S_{\mathrm{bulk}}(\varphi)=\lim_{r\rightarrow 0}\tilde{S}_{l}(\varphi),$$
where $r=\max\{\mathrm{diam}(l_{1}),\dots,\mathrm{diam}(l_{n})\}$.
\end{remark}
Now we define the regularized action\footnote{It should be always clear from the context for which space the action $S$ stands for.} as
\begin{equation}\label{e:S-reg}
S=S(D;w_{1},\dots,w_{n})=S_{\mathrm{bulk}}(\varphi)+\frac{\sqrt{-1}}{2}\sum_{k=2}^{g}\int_{C_{k}}\theta_{L_{k}^{-1}}(\varphi).
\end{equation}

This completes the definition of $S$ provided that fundamental domain $D$ is such that $w_{1},\dots,w_{n}\in  \mathrm{Int}\,D$. As in the compact case, $S$ does not depend on the choice of $D$ with the property that $w_{1},\dots,w_{n}\in \mathrm{Int}\, D$. However, $S(D;w_{1},\dots,w_{n})$ depends on the choice of representatives in $\Sigma\cdot\{w_{1},\dots,w_{n}\}$ and no longer determines a function oh the Schottky space $\mathfrak{S}_{g,n}$. Its geometric meaning is the following (cf.~Lemma \ref{h tautological}).

\begin{lemma} \label{transform} The regularized Liouville action determines a Hermitian metric $\exp\{S/\pi\}$ in the holomorphic line bundle $\cL=\cL_{1}\otimes\cdots\otimes\cL_{n}$
over $\mathfrak{S}_{g,n}$.
\end{lemma}
\begin{proof} It is sufficient to prove that for $i=1,\dots,n$,
$$S(\tilde{D};w_{1},\dots,L_{k}w_{i},\dots,w_{n})-S(D;w_{1},\dots,w_{n})=\pi\log|L^{\prime}_{k}(w_{i})|^{2},$$
where $w_{1},\dots,w_{n}\in \mathrm{Int}\,D$ and $w_{1},\dots,w_{i-1},L_{k}w_{i},w_{i+1},\dots,w_{n}\in\mathrm{Int}\,\tilde{D}$. Moreover, it is sufficient to consider
 the case when
$$\tilde{D}=(D\setminus D_{0})\cup L_{k}(D_{0})$$
and $D_{0}\subset D$ is such that $\del D_{0}\cap\del D\subset C_{k}$ and $w_{i}\in D_{0}$, while all other $w_{j}\in D\setminus D_{0}$, $j\neq i$. Indeed, any choice of a fundamental domain for $\Sigma$ is obtained from $D$ by a finite combination of such transformations.

Put
\begin{equation} \label{S-delta}
I_{\delta}(D;w_{1},\dots,w_{n})=\iint_{D_{\delta}}\omega(\varphi) +\sum_{k=2}^{g}\int_{C_{k}}\theta_{L_{k}^{-1}}(\varphi).
\end{equation}
Since $\tilde{C}_{j}=C_{j}$ for $j\neq k$ and $\tilde{C}_{k}=C_{k}-\del D_{0}$, we have
\begin{align*}
\Delta I_{\delta}&=I_{\delta}(\tilde{D};w_{1},\dots,L_{k}w_{i},\dots,w_{n})-I_{\delta}(D;w_{1},\dots,w_{n})\\
&=\iint_{L_{k}(D_{0})\setminus \tilde{D}_{i}(\delta)}\omega(\varphi) - \iint_{D_{0}\setminus D_{i}(\delta)}\omega(\varphi)-\int_{\del D_{0}}\theta_{L_{k}^{-1}}(\varphi).
\end{align*}
It follows from \eqref{phi-transform} that
\begin{align*}
L_{k}^{\ast}(\omega(\vphi))=\omega(\vphi)\circ L_{k}|L_{k}'|^{2}=\omega(\vphi)+d\theta_{L_{k}^{-1}}(\vphi),
\end{align*}
and by Stokes theorem we get
\begin{align*}
\Delta I_{\delta} & =\iint_{D_{0}\setminus L_{k}^{-1}(\tilde{D}_{i}(\delta))}L_{k}^{\ast}(\omega(\vphi)) - \iint_{D_{0}\setminus D_{i}(\delta)}\omega(\vphi)-\int_{\del D_{0}}\theta_{L_{k}^{-1}}(\vphi)\\
&=\iint_{D_{0}\setminus L_{k}^{-1}(\tilde{D}_{i}(\delta))}\omega(\vphi) - \iint_{D_{0}\setminus D_{i}(\delta)}\omega(\vphi)-\int_{\del L_{k}^{-1}(\tilde{D}_{i}(\delta))}\theta_{L_{k}^{-1}}(\vphi)\\
&=\iint_{D_{0}\setminus D(\tilde\delta)}\omega(\vphi) - \iint_{D_{0}\setminus D_{i}(\delta)}\omega(\vphi) +o(1)\quad\text{as}\quad\delta\rightarrow 0,
\end{align*}
where $\tilde\delta=\delta/|L'_{k}(w_{i})|$. Thus for $|L_{k}'(w_{i})|<1$ we have
$$\Delta I_{\delta}=-\iint_{K_{i}}|\varphi_{w}|^{2}dw\wedge d\bar{w}+o(1),$$
where $K_{i}$ is the annulus  $\delta\leq |w-w_{i}|\leq \tilde\delta$.
It now follows from \eqref{as-1} that
$$\Delta I_{\delta}=-4\pi\sqrt{-1}\log|L_{k}'(w_{i})|+o(1).$$

In case $|L_{k}'(w_{i})|>1$ we have
$$\Delta I_{\delta}=\iint_{\tilde{K}_{i}}|\varphi_{w}|^{2}dw\wedge d\bar{w}+o(1)=-4\pi\sqrt{-1}\log|L_{k}'(w_{i})|+o(1),$$
where $\tilde{K}_{i}$ is the annulus  $\tilde\delta\leq |w-w_{i}|\leq\delta$.
 \end{proof}

Combining with Lemma \ref{h tautological} we obtain
\begin{corollary} \label{Combination} Put $H=h_{1}\cdots h_{n}$. Then
\begin{align} \label{L-Schottky}
\cS =S -\pi\log H
\end{align}
determines a smooth real-valued function on $\mathfrak{S}_{g,n}$.
\end{corollary}
\begin{remark}\label{delta-i} Let $D(w_{i};\delta_{i})=\{w\in\C: |w-w_{i}|<\delta_{i}\}$, where $\delta_{i}=|a_{i}(1)|\delta$. Since $a_{i}(1)\mapsto L'_{k}(w_{i})a_{i}(1)$ under the transformation $w_{i}\mapsto L_{k}(w_{i})$, we have that up to $O(\delta^{2})$ terms $D(L_{k}w_{i};\delta_{i})=L_{k}\left(D(w_{i};\delta_{i})\right)$. This shows that \eqref{L-Schottky} can be also defined as
\begin{align} \label{L-Schottky-2}
\cS =\lim_{\delta\rightarrow 0^{+}}&\left(\frac{\sqrt{-1}}{2}\iint_{D_{\delta}(h)}\omega(\varphi)
+2\pi n(\log\delta + 2\log\left|\log\delta\right|)\right) \\
&\hspace{3.5cm}+\frac{\sqrt{-1}}{2}\sum_{k=2}^{g}\int_{C_{k}}\theta_{L_{k}^{-1}}(\varphi), \notag
\end{align}
where $D_{\delta}(h)=D\setminus \cup_{i=1}^{n}D(w_{i}; \delta_{i})$.
\end{remark}

\section{Potentials for the WP and TZ metrics}\label{s:potential}
Here using first Fourier coefficients of Klein's Hauptmodul we construct a global potential for the TZ metric on  $\cM_{0,n}$.
For the Schottky space $\mathfrak{S}_{g,n}$ we prove
that the first Chern forms of the line bundles $\cL_{i}$ with Hermitian metrics $h_{i}$ are $\dfrac{4}{3}\omega_{\mathrm{TZ},i}$. We also prove  that $\dfrac{1}{\pi^{2}}\omega_{\mathrm{WP}}$ is the first Chern form of the line bundle $\cL=\cL_{1}\otimes\cdots\otimes\cL_{n}$ with the Hermitian metric $\exp\{S/\pi\}$, where $S$ is the regularized classical Liouville action \eqref{L-Schottky}. As a corollary, the following combination $\omega_{\mathrm{WP}}-\dfrac{4\pi^{2}}{3}\omega_{\mathrm{TZ}}$ of WP and TZ metics has a global \Ka potential on $\mathfrak{S}_{g,n}$.
\subsection{Potential for the TZ metric on $\cM_{0,n}$} \label{TZ potential}
As in Sect.~\ref{genus 0}, let $\Gamma$ be marked normalized Fuchsian group of type $(0,n)$ uniformizing the Riemann surface $X=\C\setminus\{w_{1},\dots,w_{n-3},0,1\}$, let $J:\HH\rightarrow X$ be the normalized covering map, and let $h_{i}=|a_{i}(1)|^{2}$, $i=1,\dots,n-1$, and $h_n=|a_n(-1)|^2$ be smooth positive functions on $\cM_{0,n}$. According to Remark \ref{h-formula} we have
\begin{align*}
\log h_{i} & =\lim_{w\rightarrow w_{i}}\left(\log|w-w_{i}|^{2}+\frac{2e^{-\varphi(w)/2}}{|w-w_{i}|}\right),\quad i=1,\dots,n-1,\\
\intertext{and}
\log h_{n} & =\lim_{w\rightarrow\infty}\left(\log|w|^{2}-\frac{2e^{-\varphi(w)/2}}{|w|}\right),
\end{align*}
where the last formula follows from \eqref{as-2}.

\begin{lemma} \label{TZ-del} We have for all $i=1,\dots,n$,
$$h^{-1}_{i}\frac{\del h_{i}}{\del w_{k}}=\dot{F}^{k}_{w}(w_{i}),\quad k=1,\dots,n-3.$$

\end{lemma}
\begin{proof}
For given $X=\C\setminus\{w_{1},\dots,w_{n-3},0,1\}\simeq\Gamma\bk\HH$ there is an isomorphism $T_{g,n}\simeq T(\Gamma)$ (see Sect.~\ref{complex}). Consider first the case $i=n$. According to \eqref{motion}, it is sufficient to show that
$$\left.\left(\frac{\del \log h^{\vep\mu}_{n}}{\del\vep}\right)\right|_{\vep=0}=\dot{F}_{w}^{k}(\infty),\quad\text{where}\quad\mu=\mu_k.$$
Using that $F^{\vep\mu}$ is holomorphic in $\vep$ at $\vep=0$ and formulas \eqref{as-2}, \eqref{F-limits-2}, \eqref{A2-formula}, we get
\begin{align*}
\begin{split}
&\left.\left(\frac{\del h^{\vep\mu}_{n}}{\del\vep}\right)\right|_{\vep=0}\\
 & =\lim_{w\rightarrow\infty}\left\{\!\left.\left(\frac{\del}{\del\vep}\right)\right|_{\vep=0}\left(\log|F^{\vep\mu}|^{2}-2(F^{\vep\mu})^{\ast}(e^{-\frac{1}{2}\varphi^{\vep\mu}})\left|\frac{F^{\vep\mu}_{w}}{F^{\vep\mu}}\right|\right)(w)\!\right\}\\
& =\lim_{w\rightarrow\infty}\left( \frac{\dot{F}^{k}(w)}{w}  -\frac{e^{-\varphi(w)/2}(w\dot{F}^k_{w}(w)-\dot{F}^k(w))|w|}{w^{2}\bar{w}}\right)\\
&=\dot{F}^{k}_{w}(\infty).
\end{split}
\end{align*}
Interchanging the order of the limit $w\rightarrow\infty$ and differentiation is legitimate since convergence in the above formula and in the definition of $h_{n}$ is uniform in a neighborhood of an arbitrary point $(w_{1},\dots,w_{n-3})\in\cM_{0,n}$. The case $i\neq n$ is considered similarly.
\end{proof}

Let $\del$ and $\delb$ be, respectively, $(1,0)$ and $(0,1)$ components of the de Rham differential $d$ on $\cM_{0,n}$. We have the following result.
\begin{proposition} \label{TZ-potential} The functions $-\log h_{i}:\cM_{0,n}\rightarrow \R_{>0}$, $i=1,\dots,n-1$, and $\log h_{n}$ are \Ka potential for the $4\pi/3$ multiples of TZ metrics,
$$\delb\del \log h_{i}=-\frac{8\pi\sqrt{-1}}{3}\omega_{\mathrm{TZ},i},\quad i\neq n\quad\text{and}\quad \delb\del \log h_{n}=\frac{8\pi\sqrt{-1}}{3}\omega_{\mathrm{TZ},n}.$$
\end{proposition}
\begin{proof} First consider the case $i=n$. We need to prove that
$$\frac{\del^{2}\log h_{n}}{\del w_{j}\del\bar{w}_{k}}=\frac{4\pi}{3}\left\langle\frac{\del}{\del w_{j}},\frac{\del}{\del w_{k}}\right\rangle_{\mathrm{TZ},n},\quad j,k=1,\dots,n-3.$$
By polarization, it is sufficient to consider the case $j=k$. According to Sect.~\ref{complex}, for given $X=\C\setminus\{w_{1},\dots,w_{n-3},0,1\}\simeq\Gamma\bk\HH$ we can use the isomorphism $T_{g,n}\simeq T(\Gamma)$. Thus we need to show that
$$\left.\left(\frac{\del^{2}\log h^{\vep\mu}_{n}}{\del\vep\del\bar{\vep}}\right)\right|_{\vep=0}=\frac{4\pi}{3}\Vert\mu\Vert^{2}_{\mathrm{TZ},n},\quad\text{where}\quad\mu=\mu_k.$$
Using that $F^{\vep\mu}$ is holomorphic in $\vep$ at $\vep=0$ and formulas \eqref{F-limits-2}, \eqref{A2-formula}, \eqref{W3-formula}, \eqref{as-2}, we get
\begin{align*}
&\left.\left(\frac{\del^{2}\log h^{\vep\mu}_{n}}{\del\vep\del\bar{\vep}}\right)\right|_{\vep=0}\\
  &=\left.\left(\frac{\del^{2}}{\del\vep\del\bar{\vep}}\right)\right|_{\vep=0}\left\{\lim_{w\rightarrow\infty}\left(\log|F^{\vep\mu}|^{2}-2(F^{\vep\mu})^{\ast}(e^{-\frac{1}{2}\varphi^{\vep\mu}})\left|\frac{F^{\vep\mu}_{w}}{F^{\vep\mu}}\right|\right)(w)\right\}\\
&=-2\lim_{w\rightarrow\infty}\left\{\vphantom{ \left|\left.\frac{\del}{\del\vep}\right|_{\vep=0}\left(\frac{F^{\vep\mu}_{w}(w)}{F^{\vep\mu}(w)}\right)^{\frac{1}{2}}\right|^{2}} \frac{1}{|w|}\left.\left(\frac{\del^{2}}{\del\vep\del\bar{\vep}}\right)\right|_{\vep=0}(F^{\vep\mu})^{\ast}(e^{-\frac{1}{2}\varphi^{\vep\mu}})(w)  \right. \\
& \hspace{5cm} + \left. e^{-\frac{1}{2}\varphi(w)}\left|\left.\frac{\del}{\del\vep}\right|_{\vep=0}\left(\frac{F^{\vep\mu}_{w}(w)}{F^{\vep\mu}(w)}\right)^{\frac{1}{2}}\right|^{2}
 \right\}\\
&=\lim_{w\rightarrow\infty}\left\{\frac{1}{2}\log|w|f_{\mu\bar\mu}(J^{-1}(w)) -\frac{1}{2}e^{-\frac{1}{2}\varphi(w)}\frac{\left|w\dot{F}_{w}(w)-\dot{F}(w)\right|^{2}}{\left|w\right|^{3}}\right\}\\
&=\pi\lim_{w\rightarrow\infty} yf_{\mu\bar\mu}(z) =\frac{4\pi}{3}\Vert\mu\Vert^{2}_{\mathrm{TZ},n}.
\end{align*}

The case $i\neq n$ is considered similarly. Here
$$\lim_{w\rightarrow w_{i}}\frac{\log|w-w_{i}|}{\im(\sigma_{i}^{-1}(J^{-1}(w))}=-2\pi,$$
and we get the different sign from the case of $i=n$.
\end{proof}
\begin{remark} \label{Wol} One can also prove Proposition \ref{TZ-potential} by using Lemma \ref{TZ-del} and another Wolpert's formula
$$\left.\frac{\del}{\del\bar\vep}\right|_{\vep=0}(f^{\vep\mu})^{\ast}(\mu_{i}^{\vep\mu})(z)=-\left(\frac{\del}{\del\z} y^{2}\frac{\del }{\del\z}\right)\!f_{\mu\bar\mu_{i}}(z)$$
(see \cite[Theorem 2.9]{W86}).
\end{remark}
\begin{remark} Let $\cL_{i}$ be the tautological line bundle on $\cM_{0,n}$ --- a holomorphic line bundle dual to the vertical tangent bundle of $\cM_{0,n}$ along the fibers of the projection $p_{i}:\cM_{0,n}\rightarrow\cM_{0,n-1}$ which `forget' the marked point $z_{i}$, $i=1,\dots,n$. The line bundles
$\cL_{i}$ are holomorphically trivial over $\cM_{0,n}$ (but not over $\overline{\cM_{0,n}}$), and the functions $h_{i}$ on $\cM_{0,n}$ are
trivializations of the Hermitian metrics in $\cL_{i}$, introduced in \cite{Weng01, W07}.
\end{remark}

By Lemma \ref{Hermite metric} and Proposition \ref{TZ-potential}, we have
\begin{corollary} \label{TZ--potential} The function $-\log H=\log h_{n}-\log h_{1}-\dots -\log h_{n-1}$ is  a potential for the $4\pi/3$ multiple of the TZ metric on $\cM_{0,n}$. The first Chern form of the Hermitian line bundle $(\lambda_{0,n},H)$ over $\frak{M}_{0,n}$ is given by
$$c_{1}(\lambda_{0,n},H)=\frac{4}{3}\omega_{\mathrm{TZ}}.$$
\end{corollary}
For each marked Fuchsian group $\Ga$ denote by $r(z)$ the projection of the regular automorphic form $-\mathcal{S}(J)(z)$ of weight $4$ to the subspace of cusp forms,
$$r(z)=\sum_{i=1}^{n-3}\alpha_{i}r_{i}(z),\quad\text{where}\quad \alpha_{i}=-\iint\limits_{\Ga\bk\HH}\mathcal{S}(J)(z)\mu_{i}(z)d^{2}z.$$
According to Sect.~\ref{genus 0}, the family of cusp forms $r(z)$ for varying $\Gamma$ determines a $(1,0)$-form $r$ on $T_{0,n}$. Denote by
$\vartheta=\sum_{i=1}^{n-3}\alpha_{i}dw_{i}$ the corresponding $(1,0)$-form on $\cM_{0,n}$.
It follows from \eqref{motion}, \eqref{Schwarz-2} that $p^{\ast}(\vartheta)=r$, where $p:T_{0,n}\rightarrow\cM_{0,n}$.

Put $\cS=S-\pi \log H$. Combining Lemma \ref{TZ-del} with the proof of Theorem 1 in \cite{TZ87a} and using Proposition \ref{TZ-potential} and Theorem 2 in \cite{TZ87a}, we obtain the following result.
\begin{corollary} \label{cS} The function $\cS:\cM_{0,n}\rightarrow\R$ satisfies
$$\del\cS=2\vartheta$$
and
\begin{equation} \label{trivial}
\delb\del\cS=-2\sqrt{-1}\left(\omega_{\mathrm{WP}}-\frac{4\pi^{2}}{3}\omega_{\mathrm{TZ}}\right).
\end{equation}
\end{corollary}
\begin{remark} \label{Quillen} Since both $H$ and $\exp\{S/\pi\}$  are Hermitian metrics in the line bundle $\lambda_{0,n}$ over $\frak{M}_{0,n}$ (see Sects.~\ref{genus 0} and \ref{spheres}), we conclude that $\cS=S-\pi\log H$ determines a function on $\frak{M}_{0,n}$. 
The combination $\omega_{\mathrm{WP}}-\dfrac{4\pi^{2}}{3}\omega_{\mathrm{TZ}}$, with the overall factor $1/12\pi$, appears in the local index theorem for families on punctured Riemann surfaces for $k=0,1$ (see \cite[Theorem 1]{TZ91}). Equation  \eqref{trivial} agrees with the fact that the analog of the Hodge line bundle $\lambda_{1}$  over $\frak{M}_{0,n}$ is trivial. The function $\cS$ plays the role of the Quillen metric in $\lambda_{1}$, defined in \cite{TZ91}.
\end{remark}

\subsection{Chern forms and potential on $\frak{S}_{g,n}$} \label{SS} As in Sect.~\ref{S space}, let $X=\Sigma\bk\Omega$ be a compact Riemann surface of genus $g$ with $n$ marked points $x_{1},\dots,x_{n}$, let $\Gamma$ be a Fuchsian group of type $(g,n)$ such that $X_{0}=X\setminus\{x_{1},\dots,x_{n}\}\cong\Gamma\bk\HH$, and let $J:\HH^{\ast}\rightarrow\Omega$ be the corresponding branched covering map. Similar to the previous section, denote by $R$ the projection of the automorphic form $\mathcal{S}(J^{-1})$ of weight $4$ for $\Sigma$ to the subspace $\cH^{2,0}(\Omega_{0},\Sigma)\cong T_{0}^{\ast}\mathfrak{S}_{g,n}$. Using pairing \eqref{pair}, we get
\begin{equation*}
R(w)=\sum_{j=1}^{3g-3+n}\beta_{j}P_{j}(w),\quad\text{where}\quad\beta_{j}=(\mathcal{S}(J^{-1}), M_{j}).
\end{equation*}
Corresponding automorphic forms over each point $(\Sigma^{\mu},w_{1}^{\mu},\dots,w_{n}^{\mu})$ determine a $(1,0)$-form $\cR$ on $\mathfrak{S}_{g,n}$.

Using \eqref{S-J Schottky} we have
$$R(w)=\pi R_{0}(w) +\sum_{i=1}^{n}R_{i}(w),$$
where
$$R_{0}(w)=-\sum_{j=1}^{3g-3+n}c_{j}P_{j}(w),\quad R_{i}(w)=\sum_{j=1}^{3g-3+n}(\cE_{i},M_{j})P_{j}(w).$$

In the next theorem, using identification of cotangent spaces to $\mathfrak{S}_{g,n}$ at each point $(\Sigma^{\mu},w_{1}^{\mu},\dots,w_{n}^{\mu})$  with $\cH^{2,0}(\Omega_{0}^{\mu},\Sigma^{\mu})$ (see Sect.~\ref{S space}), we explicitly describe canonical connections on the Hermitian line bundles $\cL_{i}$ and $\cL$.

\begin{theorem} \label{L-TZW} Let $\del$ and $\delb$ be $(1,0)$ and $(0,1)$ components of de Rham differential on $\mathfrak{S}_{g,n}$. The following statement holds.
\begin{itemize}
\item[(i)] In a local holomorphic frame canonical connection on the Hermitian line bundle $(\cL_{i},h_{i})$ is given by
$$h^{-1}_{i}\del h_{i}=-\frac{2}{\pi}R_{i},\quad i=1,\dots,n.$$
\item[(ii)] In a local holomorphic frame canonical connection on the Hermitian line bundle $(\cL,\exp\{S/\pi\})$ is given by
$$\frac{1}{\pi}\del S=2R_{0}.$$
\item[(iii)] The function $\cS: \mathfrak{S}_{g,n}\rightarrow\R$ given by \eqref{L-Schottky} satisfies
\end{itemize}
$$\del\cS=2\cR.$$
\end{theorem}
\begin{proof} To prove part (i), it is sufficient to show that
$$\left.\left(\frac{\del \log h^{\vep\mu_{i}}_{j}}{\del\vep}\right)\right|_{\vep=0}=-\frac{2}{\pi}(\cE_{j},M_{i}).$$
Repeating verbatim computation in the proof of Lemma \ref{TZ-del} we get
$$\left.\left(\frac{\del\log h^{\vep\mu_{i}}_{j}}{\del\vep}\right)\right|_{\vep=0}=\dot{F}^{i}_{w}(w_{j}).$$
Now using \eqref{dot F} and \eqref{E-series-S} we obtain
$$\pi\dot{F}^{i}_{w}(w_{j})=-\iint_{\C}M_{i}(w)\left(\frac{1}{(w-w_{j})^{2}}-\frac{1}{w(w-1)}\right)d^{2}w=-2(\cE_{j},M_{i}),$$
and the result follows.

To prove part (ii), it is sufficient to show that
$$\left.\frac{\del}{\del\vep}\right|_{\vep=0}S(\Sigma^{\vep\mu_{i}}; w_{1}^{\vep\mu_{i}},\dots,w_{n}^{\vep\mu_{i}})=-2\pi c_{i},\quad i=1,\dots,3g-3+n.$$
We have
$$\cI=\left.\frac{\del}{\del\vep}\right|_{\vep=0}S(\Sigma^{\vep\mu_{i}}; w_{1}^{\vep\mu_{i}},\dots,w_{n}^{\vep\mu_{i}})=\ihalf\lim_{\delta\rightarrow 0}\left.\frac{\del}{\del\vep}\right|_{\vep=0}I_{\delta}(\vep)$$
and
\begin{gather*}
I_{\delta}(\vep)
=\iint_{D^{\vep\mu_{i}}_{\delta}}\omega(\varphi^{\vep\mu_{i}}) +\sum_{k=2}^{g}\int_{C^{\vep\mu_{i}}_{k}}\theta_{(L^{\vep\mu_{i}}_{k})^{-1}}(\varphi^{\vep\mu_{i}}).
\end{gather*}

The calculation of $\cI$ almost verbatim repeats the corresponding computation in the proof of Theorem 1 in \cite{TZ87b}, where regularization at the punctures is treated as in the proof of Theorem 1 in \cite{TZ87a}. Namely, using commutative diagram \eqref{CD-S} and the change of variables $w\mapsto F^{\vep\mu_{i}}(w)$, we get
\begin{gather*}
I_{\delta}(\vep)=\iint_{D_{\delta}(\vep)}(F^{\vep\mu_{i}})^{\ast}(\omega(\varphi^{\vep\mu_{i}})) +\sum_{k=2}^{g}\int_{C_{k}}(F^{\vep\mu_{i}})^{\ast}(\theta_{(L_{k}^{\vep\mu_{i}})^{-1}}(\varphi^{\vep\mu_{i}})),
\end{gather*}
where
$$D_{\delta}(\vep)=D\setminus\cup_{j=1}^{n}\left\{w\in D\,|\, |F^{\vep\mu_{i}}(w)-F^{\vep\mu_{i}}(w_{j})|<\delta\right\}.$$

To compute $\del I_{\delta}(\vep)/\del\vep|_{\vep=0}$, we need to differentiate under the integral sign as well as over the variable integration domain $D_{\delta}(\vep)$. The first computation repeats verbatim the one in \cite[Theorem 1]{TZ87b}, with the only change that now integration goes over $D_{\delta}$ and
$\del D_{\delta}$ instead of $D$ and $\del D$ as in \cite{TZ87b}. For the second contribution we use an elementary formula for differentiating a given $2$-form $\Omega$ over a smooth family of variable domains $\mathcal{D}(\vep)$,
$$\left.\frac{\del}{\del\vep}\right|_{\vep=0}\iint_{\mathcal{D}(\vep)}\Omega=\int_{\del \mathcal{D}}i_{V}(\Omega),$$
where $V$ is a vector field along $\del \mathcal{D}$ corresponding to the family of curves $\del\mathcal{D}(\vep)$. In our case we readily obtain
$$\left.\frac{\del}{\del\vep}\right|_{\vep=0}\iint_{D_{\delta}(\vep)}\omega=-\sum_{j=1}^{n}\int_{\del D_{j}(\delta)}|\varphi_{w}|^{2}\left(\dot{F}^{i}(w)-\dot{F}^{i}(w_{j})\right)d\bar{w},$$
where $\del D_{j}(\delta)$ are oriented as a boundary of  $D_{j}(\delta)$ (which is opposite to the orientation from $\del D_{\delta}$).

Thus as in \cite{TZ87b} we get
\begin{gather*}
\left.\frac{\del I_{\delta}(\vep)}{\del\vep}\right |_{\vep=0}=2\iint_{D_{\delta}}\mathcal{S}(J^{-1})(w)M_{i}(w) dw\wedge d\bar{w}\\ -\sum_{j=1}^{n}\int_{\del D_{j}(\delta)}|\varphi_{w}|^{2}\left(\dot{F}^{i}(w)-\dot{F}^{i}(w_{j})\right)d\bar{w} + I_{1}+ I_{2}+I_{3},
\end{gather*}
where
$$I_{1}=-2\sum_{j=1}^{n}\int_{\del D_{j}(\delta)}\varphi_{w}\dot{F}^{i}_{\bar{w}}d\bar{w},\quad I_{2}=-\sum_{j=1}^{n}\int_{\del D_{j}(\delta)}\varphi_{w}\dot{F}^{i}_{w}dw,$$
$$I_{3}=-\sum_{j=1}^{n}\int_{\del D_{j}(\delta)}\varphi_{\bar{w}}\dot{F}^{i}_{w}d\bar{w}.$$
As in the proof of Theorem 1 in \cite{TZ87a}, we obtain that $I_{1}, I_{2}$ and $I_{3}$ are $o(1)$ as $\delta\rightarrow 0$.
Also,
\begin{gather*}
\lim_{\delta\rightarrow 0}\iint_{D_{\delta}}\mathcal{S}(J^{-1})(w)M_{i}(w) dw\wedge d\bar{w}=-2\sqrt{-1}(\mathcal{S}(J^{-1}),M_{i})=-2\sqrt{-1}\beta_{i},
\end{gather*}
and it follows from asymptotic behavior \eqref{as-1} that
\begin{gather*}
\lim_{\delta\rightarrow 0}\sum_{j=1}^{n}\int_{\del D_{j}(\delta)}|\varphi_{w}|^{2}(\dot{F}^{i}(w)-\dot{F}^{i}(w_{j}))d\bar{w}\\
=\lim_{\delta\rightarrow 0}\sum_{j=1}^{n}\int_{\del D_{j}(\delta)}\left(\frac{ \dot{F}^{i}(w)-\dot{F}^{i}(w_j)}{\left|w-w_j\right|^2}+
\frac{2(\dot{F}^{i}(w)-\dot{F}^{i}(w_i)) }{|w-w_j|^2\log|w-w_j|}\right) d\bar{w}\\
=2\pi\sqrt{-1}\sum_{j=1}^{n}\dot{F}^{i}_{w}(w_{j}).
\end{gather*}
Thus we have
$$\cI=2\beta_{i}+\pi\sum_{j=1}^{n}\dot{F}^{i}_{w}(w_{j})=-2\pi c_{i}.$$

Part (iii) immediately follows from (i) and (ii).
\end{proof}
\begin{remark} One can also restate the proof using cohomological methods developed in \cite{TT03}.
\end{remark}
\begin{theorem} \label{S-TZWP} The following statements hold.
\begin{itemize}
\item[(i)] The first Chern form of the Hermitian line bundle $(\cL_{i},h_{i})$ is given by
$$c_{1}(\cL_{i},h_{i})=\frac{4}{3}\omega_{\mathrm{TZ},i},\quad i=1,\dots,n.$$
\item[(ii)] The first Chern form of the Hermitian line bundle $(\cL,\exp\{S/\pi\})$ is given by
$$c_{1}(\cL,\exp\{S/\pi\})=\frac{1}{\pi^{2}}\omega_{\mathrm{WP}}.$$
\item[(iii)] The function $\cS$ given by \eqref{L-Schottky} satisfies
$$\delb\del \cS=-2\sqrt{-1}\left(\omega_{\mathrm{WP}}-\frac{4\pi^{2}}{3}\omega_{\mathrm{TZ}}\right),$$
i.e., $-\cS$ is a potential for this special combination of WP and TZ metrics.
\end{itemize}
\end{theorem}
\begin{proof} Since
$$c_{1}(\cL_{i},h_{i})=\frac{\sqrt{-1}}{2\pi}\delb\del\log h_{i},$$
the proof of part (i) is exactly the same as that of Proposition \ref{TZ-potential}. Using part (ii) of Theorem \ref{L-TZW}, we obtain the proof of part (ii)  by repeating, line by line, the computation in \cite[Theorem 2]{TZ87b}. Part (iii) immediately follows from (i) and (ii).
\end{proof}
\begin{remark} As in case of the moduli space $\frak{M}_{0,n}$ (see Remark \ref{Quillen}), the combination $\omega_{\mathrm{WP}}-\dfrac{4\pi^{2}}{3}\omega_{\mathrm{TZ}}$, with the overall factor $1/12\pi$, appears in the local index theorem for families on punctured Riemann surfaces for $k=0,1$ (see \cite[Theorem 1]{TZ91}). Part (iii) of Theorem \ref{S-TZWP} agrees with the fact that the Hodge line bundle $\lambda_{1}$ is holomorphically trivial over $\mathfrak{S}_{g,n}$. 
It would be interesting to relate the function $\cS$ with the Quillen metric in $\lambda_{1}$, defined in \cite{TZ91} (see \cite[\S3]{Z89}).
\end{remark}
\begin{remark} Let $\cM_{g,n}$ be the moduli space of $n$-pointed algebraic curves of genus $g$. The Hermitian metrics $h_{i}$ in the line bundles $\cL_{i}$ provide explicit expressions for the pullbacks of the Hermitian metrics in tautological line bundles over $\cM_{g,n}$, introduced in \cite{Weng01,W07}.
\end{remark}

\section{Generalization to quasi-Fuchsian deformation spaces}\label{s:q-Fuchsian}
Here we define the Liouville action functional on the quasi-Fuchsian deformation spaces of punctured Riemann surfaces and prove that it is a \Ka potential for the Weil-Petersson metric. The construction follows very closely our work \cite{TT03} for compact Riemann surfaces, so here we just highlight the necessary modifications and refer to  \cite{TT03} for the details. For the convenience of the reader here we are using the same notations as in \cite{TT03}.

Let $\Gamma$ be a marked, normalized, quasi-Fuchsian group of type $(g,n)$ such that $3g-3+n>0$. Its region of discontinuity $\Omega$ has two invariant components $\Omega_1$ and $\Omega_2$ separated by a quasi-circle $\mathcal{C}$. There exists a quasiconformal homeomorphism $J_1$ of $\hat{\mathbb{C}}$ such that
\begin{enumerate}
\item[\textbf{QF1}]$J_1$ is holomorphic on $\mathbb{U}$ and $J_1(\mathbb{U})=\Omega_1$, $J_1(\mathbb{L})=\Omega_2$, $J_1(\mathbb{R})=\mathcal{C}$, where $\mathbb{U}$ and
$\mathbb{L}$ are, respectively, upper and lower half-planes.
\item[\textbf{QF2}] $J_1$ fixes $0, 1$ and $\infty$.
\item[\textbf{QF3}] $ \Gamma_1=J_1^{-1}\circ \Gamma\circ J_1$ is a marked, normalized Fuchsian group.
\end{enumerate}

Let $X\simeq\Gamma\backslash\Omega_1$ and $Y\simeq \Gamma\backslash\Omega_2$ be
corresponding marked punctured Riemann surface of type $(g,n)$ with opposite orientations.
There is also a quasiconformal homeomorphism $J_2$ of $\hat{\mathbb{C}}$, holomorphic on $\mathbb{L}$ with a Fuchsian group
${\Gamma}_2= J_2^{-1}\circ \Gamma\circ J_2$ so that $X\simeq \Gamma_1\backslash\mathbb{U}$ and $Y\simeq {\Gamma}_2\backslash \mathbb{L}$.
The hyperbolic metric $e^{\phi_{\mathrm{hyp}}(w)}|dw|^2$ on $\Omega=\Omega_1\sqcup\Omega_2$ is explicitly given by
\begin{align}\label{e:metric}
e^{\phi_{\mathrm{hyp}}(w)}= \frac{|(J_i^{-1})_w(w)|^2}{| \im (J^{-1}_i(w))|^2} \quad \text{if} \quad w\in \Omega_i, \ \ i=1,2,
\end{align}
and is a pull-back by the map $J^{-1}:\Omega_1\sqcup\Omega_2\rightarrow\mathbb{U}\sqcup\mathbb{L} $ of the hyperbolic metric on $\mathbb{U}\sqcup\mathbb{L}$, where $J|_{\mathbb{U}}=J_{1}|_{\mathbb{U}}$ and
$J|_{\mathbb{L}}=J_{2}|_{\mathbb{L}}$.

Denote by $\frak{D}(\Gamma)$ the deformation space of the quasi-Fuchsian group $\Gamma$. It is a complex manifold of complex dimension $6g-6+2n$ with the Weil-Petersson \Ka form
(see \cite[Sect. 3]{TT03} and references therein). As in \cite{TT03}, we define the smooth function $S:\mathfrak{D}(\Gamma)\rightarrow\R$, the critical value of the Liouville action functional, using homology and cohomology double complexes associated with the $\Gamma$-action on $\Omega$.
\subsection{Homology construction}

Start with marked normalized Fuchsian group $\Gamma$ of type $(g,n)$ with $2g$
 hyperbolic generators $\alpha_1,\dots,\alpha_g, \beta_1,\dots ,
\beta_g$ and $n$ parabolic generators $\lambda_1, \ldots,
\lambda_n$ satisfying the single relation
\begin{equation*}
\gamma_1\cdots\gamma_g\lambda_1\cdots\lambda_n=\id,
\end{equation*}
where
$\gamma_k=[\alpha_k,\beta_k]=\alpha_k\beta_k\alpha_k^{-1}\beta_k^{-1}$.
Here the attracting and
repelling fixed points of $\alpha_1$ are, respectively, $0$ and
$\infty$, and the attracting fixed point of $\beta_1$ is $1$.

The double homology complex $\mathsf{K}_{\bullet,\bullet}$ is defined as $\SSS_{\bullet}\otimes_{\Z\Gamma}\BBB_{\bullet}$, a tensor product over the integral group ring $\Z\Gamma$, where $\SSS_{\bullet}=\SSS_{\bullet}(\mathbb{U})$ is the singular chain
complex of $\mathbb{U}$ with the differential $\del'$, considered as a right $\Z\Gamma$-module, and $\BBB_{\bullet}=\BBB_{\bullet}(\Z\Gamma)$ is the standard bar resolution complex for $\Gamma$ with the differential $\del''$.
The associated total complex $\mathrm{Tot}\,\KKK$ is equipped with the total differential $\del=\del'+(-1)^{p}\del''$ on  $\mathsf{K}_{p,q}$.

The analog of the total $2$-cycle
that represents the fundamental class of the compact Riemann surface in \cite[Sect. 2.2.1]{TT03} is the following $2$-chain
$$\Sigma=F+L-V,$$
satisfying
\begin{equation} \label{cycle}
\del\Sigma =-\sum_{i=1}^n z_i\otimes[\lambda_i],
\end{equation}
where $z_{i}\in\R$ are fixed points of the parabolic generators $\lambda_{i}$.
\begin{remark} Note that $\Sigma-\bar\Sigma$, where $\bar\Sigma=\bar{F}+\bar{L}-\bar{V}$, is a total $2$-cycle in the double complex associated with $\mathbb{U}\sqcup\mathbb{L}$,
$$\del(\Sigma-\bar\Sigma)=0.$$

\end{remark}
Here the elements $F\simeq F\otimes \left[\;\right]\in \KKK_{2,0}$, $L\in\KKK_{1,1}$ and $V\in\KKK_{0,2}$ are defined as follows.
The element $F$ is a standard  fundamental domain for $\Gamma$ in  $\mathbb{U}$ --- a closed non-Euclidean polygon with
$4g+2n$ edges labeled by $a_k$, $a_k'$, $b_k'$, $b_k$, $k=1,\dots,g$, and
$c_i$, $c_i'$, $i=1,\dots, n$, satisfying
$ \alpha_k(a_k')=a_k$, $\beta_k(b_k')=b_k$ and
$\lambda_{i}(c_i')=c_{i}$. The orientation of the edges
is  such that
\begin{equation*}
\pa' F=\sum_{k=1}^{g}(a_k+b_k'-a_k'-b_k)+\sum_{i=1}^{n}(c_{i}-c_{i}').
\end{equation*}
Set $\pa' a_k=a_k(1)-a_k(0)$, $\pa' b_k=b_k(1)-b_k(0)$,
$\pa' c_i=c_i(1)-c_i(0)$, so that $a_k(0)=b_{k-1}(0)$, $k=2,\dots, g$, $a_1(0)=c_n'(0) $, $c_i(0)=c_{i-1}'(0)$,
$i=2,\dots,n $, $c_1(0)= b_g(0)$. The elements $L\in\KKK_{1,1}$ and $V\in\KKK_{0,2}$ are given by
\begin{equation} \label{L}
L = \sum_{k=1}^g (b_k\otimes\left[\beta_k\right] -a_k\otimes\left[\alpha_k\right])-\sum_{i=1}^n
c_i\otimes\left[\lambda_i\right]
\end{equation}
and
\begin{equation} \label{V}
\begin{split}
V &= \sum_{k=1}^g \left(a_k(0)\otimes\left[\alpha_k|\beta_k\right] -
b_k(0)\otimes\left[\beta_k|\alpha_k\right] +
b_k(0)\otimes\left[\gamma_k^{-1}|\alpha_k\beta_k\right]\right)\\ &
-\sum_{k=1}^{g-1}
b_g(0)\otimes\left[\gamma_g^{-1}\ldots\gamma_{k+1}^{-1}|\gamma_k^{-1}\right]+\sum_{i=1}^{n-1}c_1(0)\otimes\left[\lambda_1
\cdots\lambda_i|\lambda_{i+1}\right]. \end{split}
\end{equation}
Finally let $P_k$ be $\Gamma$-contracting paths in $\mathbb{U}$ connecting 0 to $b_k(0)$ (see \cite[Definition 2.3]{TT03}),
and let
\begin{equation} \label{W}\begin{split}
W &= \sum_{k=1}^g \left(P_{k-1}\otimes\left[\alpha_k|\beta_k\right] -
P_k\otimes\left[\beta_k|\alpha_k\right] +
P_k\otimes\left[\gamma_k^{-1}|\alpha_k\beta_k\right]\right)\\ &
-\sum_{k=1}^{g-1}
P_g\otimes\left[\gamma_g^{-1}\cdots\gamma_{k+1}^{-1}|\gamma_k^{-1}\right]+\sum_{i=1}^{n-1}P_g\otimes\left[\lambda_1
\cdots\lambda_i|\lambda_{i+1}\right]. \nonumber
\end{split}\end{equation}

Now let $\Gamma$ be a marked, normalized, quasi-Fuchsian group of type $(g,n)$ with the $2g$ loxodromic generators $\alpha_1,\dots,\alpha_g, \beta_1,\dots ,
\beta_g$ and $n$ parabolic generators $\lambda_1, \ldots,\lambda_n$, and let $\Gamma_{1}$ be the Fuchsian group such that $\Gamma_1=J_{1}^{-1} \circ \Gamma \circ J_{1}$.
The double complex associated with $\Omega=\Omega_{1}\sqcup\Omega_{2}$ and the group $\Gamma$ is a push-forward by the map $J_{1}$ of the double complex
associated with $\mathbb{U}\sqcup\mathbb{L}$ and the group $\Gamma_{1}$. The corresponding total $2$-cycle for this complex is given by
$$\Sigma_{1}-\Sigma_{2}=J_{1}(\Sigma)-J_{1}(\bar\Sigma)=F_{1}-F_{2}+L_{1}-L_{2}-V_{1}+V_{2},$$
where $F_{1}=J_{1}(F)$,  $F_{2}=J_{1}(\bar{F})$,  $L_{1}=J_{1}(L)$,  $L_{2}=J_{1}(\bar{L})$,  $V_{1}=J_{1}(V)$, $V_{2}=J_{1}(\bar{V})$, and
we continue to denote by $\Sigma-\bar{\Sigma}$ the total $2$-cycle for the  double complex
associated with $\mathbb{U}\sqcup\mathbb{L}$ and the group $\Gamma_{1}$.

\subsection{Cohomology construction} The corresponding double complex in cohomology $\CCC^{\bullet,\bullet}$ is defined as
$\CCC^{p,q}=\Hom_{\C}(\BBB_{q},\AAA^{p})$, where $\AAA^{\bullet}$ is the complexified de Rham complex on $\Omega=\Omega_{1}\sqcup\Omega_{2}$. The associated total complex $\mathrm{Tot}\,\CCC$ is equipped with the total differential $D=d+(-1)^{p}\delta$ on  $\mathsf{C}_{p,q}$, where $d$ is the de Rham differential and $\delta$ is the group coboundary. The natural pairing $\langle\,,\,\rangle$ between
$\CCC^{p,q}$ and $\KKK_{p,q}$ is given by the integration over chains
(see \cite{TT03} for details).

Put $\vphi=\phi_{\mathrm{hyp}}$. As in \cite{TT03}, starting from the $2$-form
\begin{equation*}
\omega[\vphi] = \left(|\vphi_{w}|^2 +e^{\vphi}\right)dw\wedge d\bar{w}\in\CCC^{2,0}
\end{equation*}
(cf.~the corresponding $2$-form in Sect.~\ref{S domains}), one constructs the total $2$-cocycle $\Psi[\vphi]$ and defines the
the Liouville action as
\begin{equation*}
S_{\Gamma}=\frac{i}{2}\langle\Psi[\vphi],\Sigma_{1}-\Sigma_{2}\rangle,
\end{equation*}
provided that integrals over $F_{1}$ and $F_{2}$ exist (as we will show below). Moreover, $S_{\Gamma}$ does not depend on the choice of the fundamental domains $F_{1}$ and $F_{2}$ for $\Gamma$ in $\Omega_{1}$ and $\Omega_{2}$. Simplifying as in \cite[Sect. 2.3.3]{TT03}, we finally obtain
\begin{align}\label{S qf}
S_{\Gamma}=&\frac{i}{2}\left(\langle \omega[\vphi], F_{1}-F_{2}\rangle-\langle\check\theta[\vphi], L_{1}-L_{2}\rangle+\langle \check{u}, W_{1}-W_{2}\rangle\right),
\end{align}
(cf.~\cite[formula (2.27)]{TT03}). Here $W_{1}=J_{1}(W)$, $W_{2}=J_{1}(\bar{W})$, and
\begin{equation} \label{e:check-theta}
\check\theta_{\gamma^{-1}}[\vphi] = \left(\vphi -
\frac{1}{2}\log|\gamma'|^2-2\log 2-\log|c(\gamma)|^2\right) \left(\frac{\gamma''}{\gamma'} dw -
\frac{\ov{\gamma''}}{\ov{\gamma'}}d\bar{w}\right)
\end{equation}
(cf.~the corresponding $1$-form in Sect.~\ref{S domains}) and
\begin{equation} \label{e:check-u}
\begin{split}
\check{u}_{\gamma_1^{-1},\gamma_2^{-1}}= & -\left(\frac{1}{2}\log|\gamma_1'|^2+\log\frac{|c(\gamma_2)|^2}{|c(\gamma_2\gamma_1)|^2}\right)
\left(\frac{\gamma_2''}{\gamma_2'}\circ\gamma_1\, \gamma_1'\, dw -
\frac{\ov{\gamma_2''}}{\ov{\gamma_2'}}\circ\gamma_1\, \ov{\gamma_1'}\,d\bar{w}\right) \\
 & + \left(\frac{1}{2}\log|\gamma_2'\circ\gamma_1|^2+\log\frac{|c(\gamma_2\gamma_1)|^2}{|c(\gamma_1)|^2}\right) \left(\frac{\gamma_1''}{\gamma_1'} dw -
\frac{\ov{\gamma_1''}}{\ov{\gamma_1'}}d\bar{w}\right). \nonumber
\end{split}\end{equation}

Denote by $z_{1i}, z_{2i}\in\R$, $i=1,\dots,n$, the fixed points of the parabolic generators of $\Gamma_{1}$ and $\Gamma_{2}$, and by
$w_{i}=J_{1}(z_{1i})=J_{2}(z_{2i})\in\mathcal{C}$ --- the fixed points of the parabolic generators $\lambda_{i}$ of $\Gamma$.  Let $\sigma_{1i}, \sigma_{2i}\in\mathrm{PSL}(2,\R)$ be such that $\sigma_{1i}\infty=z_{1i}$ and $\sigma_{2i}\infty=z_{2i}$, $i=1,\dots,n$.
\begin{lemma}\label{l:hyp-cusp} Let $e^{\varphi(w)}|dw|^2$ be the hyperbolic metric on  $\Omega=\Omega_{1}\sqcup\Omega_{2}$. Then
\begin{align*}
(\varphi\circ J_{1}\circ \sigma_{1i})(z) & =2\log y +O(1) \quad \text{as} \quad y=\im z\rightarrow \infty,\\
(\varphi\circ J_{2}\circ \sigma_{2i})(z) & = 2\log |y| +O(1) \quad \text{as} \quad y=\im z\rightarrow -\infty.
\end{align*}
\end{lemma}
\begin{proof} It is sufficient to prove the first formula. By definition,
$$(\vphi\circ J_{1})(z)+\log |J'_{1}(z)|^{2}=-2\log y.$$
Let $\sigma_{i}\in\mathrm{PSL}(2,\C)$ be such that $\sigma_{i}(\infty)=w_{i}$.
The map $\tilde{J}_{1}=\sigma_i^{-1}\circ J_{1}\circ \sigma_{1i}$ is univalent and preserves $\infty$, so that in the neighborhood of $\infty$
\begin{align*}
\tilde{J}_{1}(z)= a_1z +a_0 + a_{-1}z^{-1} + a_{-2}z^{-2}+\ldots,\quad\text{where}\quad a_{1}\neq 0.
\end{align*}
Whence
\begin{align*}
(J_{1}\circ \sigma_{1i})(z)= w_i +b_{-1}z^{-1} +b_{-2}z^{-2} + b_{-3}z^{-3}+\ldots,\quad\text{where}\quad b_{-1}\neq 0.
\end{align*}
Thus as $y\rightarrow\infty$ we obtain
\begin{align*}
(\varphi\circ J_{1})(\sigma_{1i}z)& = -\log |J'_1(\sigma_{1i}z)|^2-\log\im(\sigma_{1i}z)^2\\
&=-\log |(J_{1}\circ\sigma_{1i})'(z)|^2-\log y^2\\
&= 4\log y -2\log y +O(1).\qedhere
\end{align*}
\end{proof}
\begin{corollary} The integrals in definition \eqref{S qf} of $S_{\Gamma}$ are convergent.
\end{corollary}
\begin{proof} Since $S_{\Gamma}$ does not depend on the choices of fundamental domains for $\Gamma$ in $\Omega_{1}$ and $\Omega_{2}$, we can choose
$F_{1}$ to be the push-forward by $J_{1}$ of a fundamental domain for $\Gamma_{1}$ in $\mathbb{U}$ and $F_{2}$ --- the push-forward by $J_{2}$ of a fundamental domain
for $\Gamma_{2}$ in $\mathbb{L}$. It immediately follows from Lemma \ref{l:hyp-cusp} that the pullback
$J_{1}^{\ast}(\omega[\vphi])$ is integrable over the fundamental domain for $\Gamma_{1}$ in $\mathbb{U}$, and $J_{2}^{\ast}(\omega[\vphi])$ --- over the fundamental domain
for $\Gamma_{2}$ in $\mathbb{L}$. The line integrals in the definition of $S_{\Gamma}$ converge as well.
\end{proof}
Using formula \eqref{S qf}, we define a function $S:\frak{D}(\Gamma)\rightarrow\R$ by setting $S(\Gamma')=S_{\Gamma'}$ for every $\Gamma'\in\frak{D}(\Gamma)$.

\subsection{Potential for the WP metric on $\frak{D}(\Gamma)$}
Let
\begin{align*}
\vartheta(z)=2\varphi_{zz}-\varphi_z^2
=&\begin{cases} 2\mathcal{S}\left(J_1^{-1}\right)(z),\quad&\text{if}\quad z\in \Omega_1\\
2\mathcal{S}\left(J_2^{-1}\right)(z),\quad&\text{if}\quad z\in \Omega_2.\end{cases}
\end{align*}
It follows from Lemma \ref{l:hyp-cusp} that an automorphic form  $\vartheta$ of weight $4$ for $\Gamma$ vanishes at the cusps $w_{1},\dots, w_{n}$. As in \cite[Sect.~4]{TT03}, the family of automorphic forms $\vartheta$ for every $\Gamma'\in\frak{D}(\Gamma)$ determines a $(1,0)$-form $\bm\vartheta$ on $\frak{D}(\Gamma)$. Denote by $d=\del+\bar\del$ the decomposition of de Rham differential on $\frak{D}(\Gamma)$ into $(1,0)$ and $(0,1)$ components.

The following result is an exact analog of Theorem 4.1 in \cite{TT03}.
\begin{theorem}\label{firstvariation} On $\frak{D}(\Gamma)$,
\begin{align*}
\del S=\bm\vartheta.
\end{align*}
\end{theorem}
The proof repeats that of Theorem 4.1 in \cite{TT03}. The only modification is a $\delta$-truncation of fundamental domains $F_{1}$ and $F_{2}$ near the cusps $w_{1},\dots,w_{n}$, needed for the application of Stokes' theorem. Lemma \ref{l:hyp-cusp} shows that in the limit $\delta\rightarrow 0$ the corresponding boundary terms vanish.

The next result is exact analog of Theorem 4.2 in \cite{TT03}
\begin{theorem} \label{qF case} The following formula holds on $\frak{D}(\Gamma)$,
\begin{align*}
d\bm\vartheta=\delb\del S=-2i\omega_{\mathrm{WP}},
\end{align*}
so that $-S$ is a \Ka potential of the WP metric on $\frak{D}(\Gamma)$.
\end{theorem}

The proof repeats that of Theorem \ref{S-TZWP} and uses Lemma \ref{l:hyp-cusp}. We leave details to the interested reader.

\section{Holography and renormalized volume} \label{Holography}

\subsection{Renormalized volume of Schottky $3$-manifolds}
Here we prove the holography principle, a precise relation between the renormalized hyperbolic volume of the corresponding Schottky $3$-manifold and the function $\cS=S-\pi \log H$, where $S$ is the regularized Liouville action and $H$ is the Hermitian metric in the line bundle $\cL$ over $\mathfrak{S}_{g,n}$ (see Sects.~\ref{S domains} and \ref{SS}). In case of the classical Liouville action on $\mathfrak{S}_{g}$, this relation was proved in \cite{Krasnov00} for classical Schottky groups and in \cite[Remark 6.2]{TT03} for the general case.

As in Sect.~\ref{S space}, let $\Sigma\subset \PSL(2,\mathbb{C})$ be marked normalized Schottky group with the region of discontinuity $\Omega\subset\hat{\mathbb{C}}$,
and let $M=\Sigma\backslash \mathbb{U}^3$ be the corresponding hyperbolic $3$-manifold with the conformal boundary at infinity $X=\Sigma\backslash \Omega$. Here $\mathbb{U}^{3}=\{(z,t) : z\in\C,\, t>0\}$ is the Lobachevsky (hyperbolic) space.

As in \cite[Sect.~5]{TT03}, let $\mathsf{K}_{\bullet,\bullet}=\SSS_{\bullet}\otimes_{\Z\Sigma}\BBB_{\bullet}$ be the corresponding double homology complex, where $\SSS_{\bullet}=\SSS_{\bullet}(\mathbb{U}^{3})$ is the singular chain
complex of $\mathbb{U}^{3}$ with the differential $\del'$ and $\BBB_{\bullet}=\BBB_{\bullet}(\Z\Sigma)$ is the standard bar resolution complex for $\Sigma$ with the differential $\del''$.

Let $R\subset\mathbb{U}^3$ be the  fundamental region for the marked Schottky group $\Sigma$ in $\mathbb{U}^3$, identified with $R\otimes[\, ]\in\mathsf{K}_{3,0}$.
We have $\del'' R=0$ and
\begin{align}
\del' R=&-D+\sum_{i=1}^g \left( H_i- L_i(H_i)\right)\label{e:del-R0}
=-D+\del'' S
\end{align}
where $D$ is the fundamental domain for $\Sigma$ in $\Omega$ as in Sect.~\ref{S space}, $H_i$ is a topological hemisphere\footnote{It is a Euclidean hemisphere when $\Sigma$ is a classical Schottky group.} with the boundary
 $C_i$, and
$S\in\mathsf{K}_{2,1}$ is defined by
\begin{align*}
S=-\sum_{i=1}^g H_i\otimes L_i^{-1}.
\end{align*}
Putting $L=\sum_{i=1}^g C_i\otimes L_i^{-1}\in\mathsf{K}_{1,1}$, we have $\del' S=-L$ and
\begin{align}\label{e:R-S}
\begin{split}
\del\left(R-S\right)=&\del' R- \del'S -\del''S\\
=&-D+\del''S+L-\del''S=-D+L.
\end{split}
\end{align}

 Let $e^{\varphi(w)} |dw|^2$ be the hyperbolic metric on  $\Omega_0=\Omega\setminus \Sigma\cdot\{w_1,\ldots, w_n\}$ (see Sect.~\ref{S space}). As in \cite[Lemma 5.1]{TT03}, there is a $\Sigma$ automorphic function $f\in C^{\infty}(\mathbb{U}^{3}\cup\Omega_{0})$ which is positive on $\mathbb{U}^{3}$ and uniformly on a compact subsets of $\Omega_{0}$ satisfies
 $$f(Z)=te^{\varphi(z)/2}+O(t^{3})\quad\text{as}\quad t\to 0,$$
 where $Z=(z,t)$. However near $(w_{i},0)$, as it follows from  \eqref{as-1}, the function $f$ satisfies
 \begin{equation} \label{f-w-i}
f(Z)=te^{\varphi(z)/2}+O(t^3|z-w_{i}|^{-2}) \qquad \text{as}\quad t\to 0,
\end{equation}
so that the level surface $f=\vep$ meets $(w_{i},0)$ and is non-compact. Hence in order to use $f$ as a level defining function for the truncated fundamental region $R\cap\{f\geq\vep\}$, one also needs to remove a neighborhoods in $\mathbb{U}^{3}$ of the points $(w_{1},0),\dots,(w_{n},0)$. Define
$$R_{\vep}=R\cap\{f\geq\vep\}\setminus\bigcup_{i=1}^{n}\left\{(z,t)\in\mathbb{U}^3\,| \, \|(z,t)-(w_{i},0)\|\leq \vep |a_{i}(1)|\right\},$$
where $\|~\|$ is the Euclidean distance in $\overline{\mathbb{U}}^{3}$
(cf. the definition of $S(\Sigma;w_1,\ldots, w_n)$ in Remark \ref{delta-i}.) As in \eqref{e:del-R0},
\begin{align}\label{e:del-R}
\del'R_\vep=-D_\vep+\sum_{i=1}^g \left( H_{i,\vep}- L_i(H_{i,\vep})\right),
\end{align}
where $D_{\vep}$ is the complement in a level surface $f(Z)=\vep$ of its intersection with $\cup_{i=1}^{n}\left\{\|Z-(w_{i},0)\|\leq \vep |a_{i}(1)|\right\}$, and $H_{i,\vep}=R_{\vep}\cap H_{i}$.

Following \cite{TT03}, we define the regularized volume of the Schottky $3$-manifold $M$ (the regularized on-shell Einstein-Hilbert action) by
\begin{equation*}
V_{\mathrm{reg}}(M)=\lim_{\vep\to 0}\left( V_\vep -\frac12 A_\vep -\frac12\pi n \left(\log\vep+2\log|\log\vep| \right) -\pi\, \chi(X)\log\vep \right),
\end{equation*}
where $V_{\vep}$ is the hyperbolic volume of $R_{\vep}$ and
$$
A_\vep=\iint_{D_\vep} dA,
$$
where $dA$ is the area form on $D_{\vep}$ induced by the hyperbolic metric on $\mathbb{U}^3$. Note that the only difference with the \cite[Def.~5.1]{TT03} is the extra subtraction of $\frac12\pi n \left(\log\vep+2\log|\log\vep| \right)$, which is due the fact that $f(Z)$ blows up as $Z\to (w_{i},0)$.

Repeating almost verbatim computations in \cite[Sect.~5.2]{TT03} and using \eqref{f-w-i}, we arrive at the following statement.
\begin{theorem} \label{S holography} Let $e^{\varphi(w)}|dw|^2$ be hyperbolic metric on $\Omega\setminus \Sigma\cdot\{w_1,\ldots, w_n\}$.  The regularized hyperbolic volume $V_{\mathrm{reg}}(M)$ of the Schottky $3$-manifold $M=\Sigma\backslash\mathbb{U}^{3}$ is well-defined and \begin{equation*}
V_{\mathrm{reg}}(M)=-\frac{1}{4}\cS+ \pi(g-1),
\end{equation*}
where $\cS$ is given by \eqref{L-Schottky}.
\end{theorem}
\begin{remark} Equivalently, the regularized volume $V_{\mathrm{reg}}(M)$ is $-1/4$ times the function $\check{\cS}=\check{S}-\pi\log H$,
where $\check{S}$ is the Liouville action without the area term.
\end{remark}

\subsection{Renormalized volume of quasi-Fuchsian $3$-manifolds} Here we define the renormalized hyperbolic volume of quasi-Fuchsian $3$-manifolds and establish its relation with the classical Liouville action in Sect.~\ref{s:q-Fuchsian}. For the renormalized hyperbolic volume, another approach to the case of geometrically finite hyperbolic 3-manifolds was developed in \cite{GMR}.

\subsubsection{Rank one cusps} Let $\Gamma$ be marked, normalized, quasi-Fuchsian group of type $(g,n)$ and let $\lambda_1,\dots,\lambda_{n}$ be its parabolic generators with fixed points $v_1,\dots,v_{n}\in\mathcal{C}$ (see Sect.~\ref{s:q-Fuchsian}). Since the stabilizer of a parabolic fixed point $v_{k}$ in $\Gamma$ is a cyclic subgroup $\langle\lambda_{k}\rangle$, it is a \emph{rank one cusp}. Denote by  $M=\Gamma\backslash \mathbb{U}^3$ the corresponding quasi-Fuchsian $3$-manifold and let $X \sqcup Y =\Gamma\backslash \Omega_1\sqcup\Omega_2$ be its conformal boundary at infinity.

If $\lambda (z)=z+1$, there exists a $s_0>0$ such that
the image of the projection $\pi:\mathbb{U}^3\to M$ of an open horoball
$$
\mathcal{H}_{s}=\{(z,t)\in\mathbb{U}^3\, |\, t> s\}
$$
is embedded into $M$ for $s\geq s_0$. In this case, $\pi(\mathcal{H}_{s})\subset M$ is homeomorphic to $\{ 0<|z|< 1\}\times \mathbb{R}$ and
$\pi(\{(z,t)\in\mathbb{U}^3\, |\, t= s\})$ corresponds to $\{ |z|=1\}\times \mathbb{R}$. The  set $\pi(\mathcal{H}_{s})$ is called a \emph{solid cusp tube}.

In general, if a rank one cusp $v=\infty$ is associated to the parabolic subgroup generated by
$\lambda=\begin{pmatrix} 1 & q\\ 0 & 1 \end{pmatrix}$, we have
\begin{align}\label{e:sigma}
\sigma^{-1} \lambda\,\sigma=\begin{pmatrix} 1 & 1 \\ 0 & 1\end{pmatrix}\qquad \text{where}\quad
\sigma=\begin{pmatrix} q^{\frac12} & 0\\ 0 & q^{-\frac12} \end{pmatrix},
\end{align}
and $\sigma$ maps an open horoball $\mathcal{H}_s$
 onto $\mathcal{H}_{|q|s}$. In this case, the corresponding solid cusp tube is $\pi(\mathcal{H}_{|q|s})$.
When a rank one cusp $v_i$ is finite and is associated with the parabolic subgroup generated by
\begin{align*}
\lambda_i=\begin{pmatrix}
1+q_iv_i & -q_i v_i^2\\ q_i & 1-q_iv_i \end{pmatrix},
\end{align*}
we have
\begin{align}\label{e:sigma2}
\sigma_i^{-1}\lambda_i \sigma_i = \begin{pmatrix} 1 & -1 \\ 0 & 1 \end{pmatrix} \qquad \text{where} \quad
\sigma_i=\begin{pmatrix} q_i^{\frac12}v_i & -q_i^{-\frac12}\\ q_i^{\frac12} & 0 \end{pmatrix}.
\end{align}
It is easy to see that $\sigma_i(\mathcal{H}_s)$ is an open horoball tangent to $\mathbb{C}$ at $\sigma_i(\infty)=v_i$, which is an Euclidean ball  with radius of
$(2|q_i|s)^{-1}$, and the corresponding solid cusp tube is $\pi(\sigma_{i}(\mathcal{H}_{s}))$. In our case the normalization of $\Gamma$ is such that all cusps are finite and the solid cusp tubes corresponding to $v_{i}$ can be chosen to be mutually disjoint in $M$. We denote $\mathcal{H}_{i,\vep}=\sigma_i(\mathcal{H}_{1/\vep^2})$, $i=1,\dots,n$.

\subsubsection{Truncation of a fundamental region}\label{ss:truncation} Let $R\subset \mathbb{U}^3$ be a fundamental region for $\Gamma$ in $\mathbb{U}^{3}$.
Put $\mathcal{H}_{\vep_0}=\cup_{i=1}^n ({\mathcal{H}}_{i,\vep_0}\setminus \overline{\mathcal{P}}_{i,\vep_0})$, where $\mathcal{P}_{i,\vep}=\sigma_i(\mathcal{P}_\vep)$ and
\begin{align*}
\mathcal{P}_\vep:=\{(z,t)\in\mathbb{U}^{3}\,|\, z=x+iy, |y| >\vep^{-1} \}.
\end{align*}
The proof of $\Gamma$-automorphic partition of unity in \cite[Lemma V.3.1]{Kra72} can be easily adapted to the case of Kleinian groups with parabolic elements. As in \cite[Lemma 5.1]{TT03}, we conclude that there exist $\vep_{0}>0$ and a $\Gamma$-automorphic function $f\in C^{\infty}(\mathbb{U}^{3}_{0})$, where $\mathbb{U}^{3}_{0}=\cup_{\gamma\in\Gamma}\gamma(R\setminus\mathcal{H}_{\vep_{0}})$, satisfying
\begin{align}\label{e:f-asymp}
f(Z)=te^{\varphi(z)/2}+O(t^3) \quad \text{as}\quad t\to 0,
\end{align}
uniformly on compact subsets of $\mathbb{U}^{3}_{0}$. Here $e^{\varphi(z)}|dz|^{2}$ is the hyperbolic metric on $\Omega_1\sqcup\Omega_2$ (see Sect.~\ref{s:q-Fuchsian}).

Using the level defining function $f$ we truncate a non-compact fundamental region $R$ as follows:
\begin{align*}R_{\vep}=R\setminus \left( \left\{Z\in R\setminus \mathcal{H}_{\vep_{0}} : f(Z)\leq \vep\right\}\, \cup\, \bigcup_{k=1}^n  \mathcal{H}_{k,\vep}\right).
\end{align*}
Similar to the Schottky case, we define a renormalized hyperbolic volume of the quasi-Fuchsian $3$-manifold $M$ by
\begin{align*}
V_{\mathrm{reg}}(M)=&\lim_{\vep\rightarrow 0} \left(V_{\vep}-\frac{1}{2}A_{\vep}-\pi\,\chi(X\sqcup Y)\log\vep\right),
\end{align*}
where $V_{\vep}$ is the hyperbolic volume of $R_{\vep}$ and $A_{\vep}$ is the area of the surface $ -F_{\vep}=\pa' R_{\vep}\cap\{f=\vep\}$ in the induced metric. Repeating computation in \cite{TT03} and analyzing the extra terms due to the removal of a solid cusp tubes from $M$, on can show that their contribution vanishes as $\vep\to 0$. Thus we arrive at the following statement.
\begin{theorem} \label{qF holography} Let $e^{\varphi(z)}|dz|^2$ be the hyperbolic metric on $\Omega_1\sqcup \Omega_2$.  The regularized hyperbolic volume $V_{\mathrm{reg}}(M)$ of the quasi-Fuchsian $3$-manifold $M=\Gamma\backslash\mathbb{U}^3$ is well-defined and
\begin{equation*}
V_{\mathrm{reg}}(M)=-\frac{1}{4}\check{S}_{\Gamma},
\end{equation*}
where $\check{S}_{\Gamma}$ is the Liouville action \eqref{S qf} without the area term,
$$\check{S}_{\Gamma}=S_{\Gamma}-\iint_{F}e^{\varphi}d^2z+4\pi \chi(X\sqcup Y)\log 2.$$
\end{theorem}
Note that in this case the statement of the theorem is exactly the same as in the compact case \cite[Theorem 5.1]{TT03}. We leave details to the interested reader.

\end{document}